\pdfmapfile{=<name>.map}
\documentclass[12pt,oneside,reqno,fleqn]{amsart}
\usepackage{autobreak}
\usepackage{bbm}
\usepackage{dsfont}
\usepackage{amssymb}
\usepackage{bbm}
\usepackage{cases}
\usepackage{amsmath}
\usepackage{graphicx}
\usepackage{mathrsfs}
\usepackage{stmaryrd}
\usepackage{color}
\usepackage{soul}
\usepackage[dvipsnames]{xcolor}
\usepackage{amsfonts}
\usepackage{amsmath,amssymb,amsthm}
\usepackage{libertine}
\usepackage[T1]{fontenc}
\usepackage[libertine,varbb]{newtxmath}
\usepackage[colorlinks,linkcolor=red,citecolor=red,urlcolor=red]{hyperref}
\usepackage[shortlabels] {enumitem}
\usepackage[nameinlink,capitalize]{cleveref}
\usepackage[spacing=true,kerning=true,tracking=true,letterspace=50]{microtype}
\hypersetup{pdftitle={A study on random dynamic system concerning singular SDEs}, pdfauthor={CL and MS}}
\numberwithin{equation}{section}
\newcommand{\be}{\begin{eqnarray}}
\newcommand{\ee}{\end{eqnarray}}
\newcommand{\ce}{\begin{eqnarray*}}
\newcommand{\de}{\end{eqnarray*}}
\newtheorem{theorem}{Theorem}[section]
\newtheorem{lemma}[theorem]{Lemma}
\newtheorem{proposition}[theorem]{Proposition}
\newtheorem{corollary}[theorem]{Corollary}
\theoremstyle{remark}
\newtheorem{assumption}[theorem]{Assumption}

\newtheorem{example}[theorem]{Example}
\newtheorem{remark}[theorem]{Remark}
\newtheorem{definition}[theorem]{Definition}

\crefname{eqn}{Equation}{Equations}
\crefname{assumption}{Assumption}{Assumptions}
\crefname{innercustomthm}{Condition}{Conditions}
\crefrangelabelformat{innercustomthm}{#3#1#4-#5#2#6}

\def\e{{\mathrm{e}}}

\def\g{\gamma}

\def\oops{\rho}
\def\dd{\mathrm{d}}

\def\tlp{\tilde {\mathbb{L}_{p}}}
\def\<{{\langle}}
\def\>{{\rangle}}
\def\({{\Big(}}
\def\){{\Big)}}

\def\bx{{\mathbf{x}}}

\def\dif{d}

\def\min{{\mathord{{\rm min}}}}

\def\={&\!\!=\!\!&}
\def\bt{\begin{theorem}}
\def\et{\end{theorem}}
\def\bl{\begin{lemma}}
\def\el{\end{lemma}}
\def\br{\begin{remark}}
\def\er{\end{remark}}
\def\bd{\begin{definition}}
\def\ed{\end{definition}}
\def\bp{\begin{proposition}}
\def\ep{\end{proposition}}
\def\bc{\begin{corollary}}
\def\ec{\end{corollary}}
\def\bx{\begin{example}}
\def\ex{\end{example}}

\def\cF{{\mathcal F}}
\def\cff{\cF}

\def\mE{{\mathbb E}}
\def\E{\mE}

\def\mP{{\mathbb P}}

\def\mR{{\mathbb R}}

\def\geq{\geqslant}
\def\leq{\leqslant}

\def\c{\mathord{{\bf c}}}

\def\tlpone{\tilde {\mathbb{L}_{p_1}}}
\newcommand{\R}{{\mathbb R}}

\newcommand{\tand}{\quad\text{and}\quad}

\newcommand{\norm}[1]{{\left\vert\kern-0.25ex\left\vert\kern-0.25ex\left\vert #1
    \right\vert\kern-0.25ex\right\vert\kern-0.25ex\right\vert}}

\renewcommand{\le}{\leq}
\renewcommand{\ge}{\geq}
\allowdisplaybreaks

\begin{document}
	\title{Expansion and attraction of RDS:  long time behavior of the solution to singular SDE}
	\date{November 28, 2022}
	\author{Chengcheng Ling \and Michael Scheutzow}
\address{Chengcheng Ling:
Technische Universit\"at Wien,
Institute of Analysis and Scientific Computing,
 1040 Wien, Austria
\\
Email: chengcheng.ling@asc.tuwien.ac.at
 }
\address{ Michael Scheutzow:
Technische Universit\"at Berlin,
Fakult\"at II, Institut f\"ur Mathematik,
10623 Berlin, Germany
\\
Email: ms@math.tu-berlin.de
 }

	\begin{abstract}
	
	We provide a framework for studying the expansion rate of the image of a bounded set under a flow in Euclidean space and apply it to stochastic differential equations (SDEs for short) with singular coefficients. If the singular drift of the SDE can be split into two terms, one of which is singular and the radial component of the other term has a radial component of sufficient strength in the direction of the origin, then the random dynamical system generated by the SDE admits a pullback attractor.\\ 
	
			\noindent {{\bf AMS 2020 Mathematics Subject Classification:} 60H10, 60G17, 60J60,  60H50}\\
			
		\noindent{{\bf Keywords:}  semi-flow,  random dynamical system, pullback attractor, singular stochastic differential equation,  Brownian motion, dispersion of random sets, chaining, Krylov estimate, regularization by noise,  elliptic partial differential equation, Zvonkin transformation} 
	\end{abstract}

\maketitle	
%\title{Existence of random attractors of the RDS generated by a SDE with singular drift}
%\author{}

\section{Introduction}
\emph{Regularization by noise}, i.e.~existence and uniqueness of solutions under the assumption of non-degenerate noise, has been established for a large class of singular stochastic differential equations (SDEs). 
It was shown recently that these equations also generate a random dynamical system (RDS), see \cite{LSV},
and like in the classical (non-singular) case it therefore seems natural to establish asymptotic properties
of these RDS for large times, like expansion rates of bounded sets and the existence of attractors 
or even {\em synchronization} (meaning that the attractor is a single random point).\\
%
%
%concerning the \emph{ well-posedness (i.e. existence and uniqueness of the solution) of the SDEs}  with %singular coefficients, especially with singular drifts, has  been studied extensively. 
%We are aware that there are few works on the study of singular SDEs in the context of \emph{random dynamical system} (RDS for short) induced initially by \cite{A} which has been well developed in past decades and contains plenty of beautiful and practical results, e.g. completeness of the flow induced by the solution to SDE \cite{LS}, existence of attractor of the flow  which was first formulated in \cite{CF}, stabilization by noise \cite{MS1},\cite{MoS1}, and most recently synchronization phenomenon in \cite{FGS}, \cite{FGS2},  \cite{GT} and \cite{SV}.  
%For singular SDEs only the existence of RDS was shown in \cite{LSV} based on the perfection procedure from \cite{AS} and \cite{KS}. In this article, by quantifying the expansion rate of the flow generated by the solution to singular SDE, we aim to study the  existence of attractor of the flow.\\

We consider an SDE on $\R^d$ with time homogeneous coefficients
\begin{align}\label{sde0}
\dif X_t=b(X_t)\,\dd t+\sigma(X_t)\,\dd W_t,\quad X_s=x\in\mathbb{R}^d,\quad t\geq s \geq 0,
\end{align}           %=(b^{(i)})_{1\leq i\leq d}
where $d\geq1$, $b: \mR^d\rightarrow\mR^d $ and $\sigma=(\sigma_{ij})_{1\leq i,j\leq d}: \mR^d\rightarrow L(\mathbb{R}^d)$ $(:=d\times d$ real valued matrices$)$ are measurable, and $(W_t)_{t\geq0}$ is a standard $d$-dimensional Brownian motion defined on some filtered probability space $(\Omega,\mathcal{F},(\mathcal{F}_t)_{t\geq 0},\mP)$. We assume that $b\in\tilde{L}_p(\mathbb{R}^d)$ (defined in \cref{notations}),  so $b$ does not have to be continuous nor bounded, and $\sigma\sigma^*$ ($\sigma^*$ denotes the matrix transpose of $\sigma$) is bounded and uniformly elliptic and $\nabla \sigma\in\tilde{L}_p(\mathbb{R}^d)$ with $p>d$ (time homogeneous \emph{Krylov-R\"ockner} condition). These are sufficient conditions for the well-posedness of the equation \eqref{sde0}, see \cite{KR} and \cite{XXZZ}. They also imply the existence of a flow  and random dynamical system (RDS) generated by the solution  to \eqref{sde0} \cite{LSV}.\\ 
%We also are aware of  that \cite{BRS} studied the RDS of \eqref{sde0} in the framework of \emph{rough path theorem}, which however is not the case here.\\

First we analyse the linear expansion rate of the flow generated by a singular SDE. In classical results,  see e.g. \cite{MS},\cite{DS},  Lipschitz continuity or one-sided Lipschitz continuity of the coefficients of the SDE is assumed to obtain bounds on the expansion rate. Obviously we lack  these properties in our current setting. Instead, we assume the noise to be non-degenerate, so we can apply the \emph{Zvonkin transformation} to get an SDE which has Lipschitz-like coefficients and this SDE is (in an appropriate sense) equivalent to the original one \eqref{sde0}. %Then the continuity issue of the coefficients is settled 
The  \emph{Zvonkin transformation} was invented by A. K. Zvonkin in \cite{Zv}  for $d=1$ and then generalized by A. Yu. Veretennikov in \cite{VE} to $d\geq 1$. It has become a rather  standard tool to study  \emph{well-posedness} of  singular SDEs, see e.g.  \cite{Zhang2011}, \cite{XieZhang2016} and \cite{XXZZ}. 
This tool heavily relies on  regularity estimates of the solution to Kolmogorov's equation corresponding to \eqref{sde0} which  can be found   for instance in \cite{Krylov} in the classical setting.  
In this paper we adapt the method to the study of the RDS induced by singular SDEs. We show that  the flow expands linearly (see \cref{2-point-sde}), a property which was established for non-singular SDEs with  not necessarily non-degenerate noise in \cite{CSS1,CSS2,LS1,LS2,MS}. 
Our proof mainly depends on stability estimates (see \cref{Transformed-two-point}). These kind of  estimates were studied before, see for instance \cite{GL}, \cite{Zhang2011} and \cite{Zhang2017}, but the dependence of the constants  on the coefficients was not specified.   
We give a formula in \cref{Transformed-two-point} which states this dependence explicitly. It also yields the expansion rate constant in \cref{2-point-sde}.

Secondly, we aim at  conditions which guarantee the existence of an  attractor for the RDS generated by a singular SDE. Clearly, one can not expect that an attractor exists   without further conditions (an example without attractor is the case in which the drift is zero and the diffusion is  constant). 
%A naive example is to take zero drift. Clearly it is impossible to show that the trajectory of Brownian diffusion has an attractor.  
Since \cite{CF}, numerous papers appeared in which the existence of attractors for various finite and infinite dimensional  RDS was shown, e.g. \cite{CS}, \cite{FGS1}, \cite{FGS2}, \cite{Gess}, \cite{GLS}, \cite{DS}, \cite{KNS} and \cite{ZZ}. 
%To be noticed that most of these RDSs are generated by certain PDEs with random forcing  like the stochastic Navier-Stokes euqations, stochastic  porous media equations, stochastic quasi-geostrophic equations. 
A common way to prove the existence of an attractor is to show the existence of a random compact \emph{absorbing} set and then to apply the criterion from \cite[Theorem 3.11]{CF}.   
Just like \cite{DS}, we will use a different and more probabilistic criterion from \cite{CDS} (Proposition \ref{flow-attractor}). Roughly speaking, all one has to show is that the image of a very large ball will be contained inside a fixed large ball after a (deterministic)  long time with high probability. In \cite{DS} this was 
shown under the assumption that the diffusion is bounded and Lipschitz and the drift $b(x)$ has a component of sufficient strength (compared to the diffusion) in the direction of the origin for large $|x|$. In our set-up, this condition is too restrictive. Instead, we assume that the drift can be written in the form $b=b_1+b_2$,
in which $b_1$ is singular and $b_2$ has a component of sufficient strength (compared to the diffusion and the localized $L_p$-norm of $b_1$) in the direction of the origin for large $|x|$.

\subsection*{Structure of the paper}
We introduce notation and the main results in \cref{notation}.
In \cref{section3} we study the expansion rate of the  diameter of the image of a bounded set under a flow under rather general conditions. These results are minor modifications of results contained in \cite{MS} which are proved by {\em chaining techniques}. \cref{section4} contains estimates on functionals of the  solution to the singular SDE, namely quantitative versions of Krylov's estimates and Khasminskii's lemma. The first part of the main results of this paper is presented in \cref{section5}, i.e. the linear expansion rate of the diameter of the image of a bounded set under the flow generated by the solution to a singular  SDE. In \cref{section6} we show the existence of an attractor of the RDS generated by the singular SDE. In \cref{A} we study  regularity estimates of elliptic partial differential equations  with emphasis on the dependence on  the coefficients. We believe that these estimates are of independent interest.  %\cref{B} contains  Khasminskii's Lemma which we crucially use in the main text. 

\section{Notation and main results} \label{notation}
\subsection{Notation}\label{notations} We denote the Euclidean norm on $\R^d$ by $|.|$ and the induced norm on $L(\R^d)$ or on $L(L(\R^d))$ by $\|.\|$. %Further, $\langle.,.\rangle$ denotes the standard inner product on $\R^d$. 
Recall that the trace of $a:=(a_{ij})_{1\leq i,j\leq d}:=\sigma\sigma^*$ satisfies $\mathrm{tr}(a)=\sum_{i,j=1}^d\sigma_{ij}^2$, where $\sigma^*$ denotes the transpose of $\sigma \in L(\R^d)$.
 For $p\in[1,\infty)$, let ${L}_p(\mathbb{\mR}^d)$ denote the space of all real Borel measurable functions on $\mathbb{R}^d$ equipped with the norm
$$\Vert f\Vert_{{L}_p}:=\Big(\int_{\mathbb{R}^d}|f(x)|^p\,\dd x\Big)^{1/p}<+\infty$$
and  $L_\infty$ denotes the space of all bounded and measurable functions equipped with the norm
$$\Vert f\Vert_\infty:=\Vert f\Vert_{L_\infty}:=\sup_{x\in\mathbb{R}^d}|f(x)|.$$
We introduce the notion of  a localized $L_p$-space for $p\in[1,\infty]$: for fixed $\delta>0$,
\begin{align}
    \label{chi}
   \tilde L_p(\mathbb{R}^d):=\{f:\Vert f\Vert_{\tilde L_p}:=\sup_{z}\Vert \xi_\delta^zf\Vert_{L_p}<\infty\},
\end{align}
 where  $\xi_\delta(x):=\xi(\frac{x}{\delta})$ and  $\xi_\delta^z(x):=\xi_\delta(x-z)$ for $x,z\in\mR^d$, $\xi\in C_c^\infty(\mR^d;[0,1])$ is a smooth function with $\xi(x)=1$ for $|x|\leq 1/2$,  and $\xi(x)=0$ for $|x|>1$. For $(\alpha,p)\in\mathbb{R}\times[1,\infty)$, let  $H^{\alpha,p}(\mR^d)$  be the usual Bessel potential space with norm 
 $$\Vert f\Vert_{H^{\alpha,p}}:=\Vert(\mathbbm{I}-\Delta)^{\alpha/2} f\Vert_{L_p},$$
 where $(\mathbbm{I}-\Delta)^{\alpha/2} f$ is defined via Fourier’s transform 
 $$(\mathbbm{I}-\Delta)^{\alpha/2} f:= \mathcal{F}^{-1}((1+|\cdot|^2)^{\alpha/2}\mathcal{F}f).$$
 The localized $H^{\alpha,p}$-space is defined as
 \begin{align*}
     \tilde H^{\alpha,p}:=\{f:\Vert f\Vert_{\tilde H^{\alpha, p}}:=\sup_{z}\Vert \xi_\delta^zf\Vert_{ H^{\alpha, p}}<\infty\}.
 \end{align*} 
From \cite[Section 2]{XXZZ} and \cite[Proposition 4.1]{Zhangzhao-NS} we know that the space $\tilde H^{\alpha,p}$ does not depend on the choice of $\xi$ and $\delta$, but the norm does, of course. More precisely, by \cite[Proposition 4.1]{Zhangzhao-NS}, for the $\tilde L_p$-norms with different $\delta$, say $\delta_1$ and $\delta_2$ and $\delta_1<\delta_2$,  if we use the notation $(\tilde L_{p})_\delta $ to denote the  $\tilde L_p$ space with support radius $\delta$ for localization, then 
\begin{align}\label{ineq:norms}
N_1\Vert\cdot\Vert_{(\tilde L_p)_{\delta_1}}\leq \Vert\cdot\Vert_{(\tilde L_p)_{\delta_2}}\leq N_2\Big(\frac{\delta_2}{\delta_1}\Big)^d \Vert\cdot\Vert_{(\tilde L_p)_{\delta_1}},
\end{align}
where $N_1, N_2$ are  constants independent of $\delta_1,\delta_2$. For convenience we take $\delta=1$ in the following. For further properties of these spaces we refer to \cite{XXZZ}. In the following, all derivatives should be interpreted in the weak sense. Occasionally we will use Einstein's summation convention (omitting the summation sign for indices appearing twice). We will often use the notation $r_+=\max\{r,0\}$ for 
the positive part of $r\in\mR$, $a\vee b:=\max\{a,b\}$ and $a\wedge b:=\min\{a,b\}$. 
    
%    For a topological space $X$, $\mathcal{B}(X)$ denotes its Borel-$\sigma$-algebra. 
    
 \subsection{Preliminaries} 
 In the following, all random processes will be defined on a given probability space
$(\Omega,\cF,\mP)$.
\begin{definition}
    A  flow $\phi$ on a Polish (i.e.~separable and completely metrizable) space $X$ equipped with its Borel-$\sigma$-algebra $\mathcal{X}=\mathcal{B} (X)$
   is a  measurable map
    \begin{align*}
        \phi: \left\{ (s,t,x,\omega) \in [0, \infty)^2\times X \times \Omega :  s \leq t < \infty \right\} \rightarrow X
    \end{align*}
    such that, for each $\omega \in \Omega$,
    \begin{enumerate}
     \item[(1)] $\phi_{s,s} (x) = x$ for all $x \in X$ and $s\geq 0$,
        \item[(2)] $(s,t,x) \mapsto \phi_{s,t}(x)$ is continuous,
        \item[(3)] for each $s,t$, the map $x\mapsto \phi_{s,t}(x)$ is one-to-one,
        \item[(4)] for all $0\leq s \leq t < u$ and $x \in X$,   the following identity holds
        \begin{align*}
            \phi_{s,u}(x) = \phi_{t,u} (\phi_{s,t} (x)).
        \end{align*}
       
    \end{enumerate}
\end{definition}
Next, we define the concepts of a metric dynamical system and a  random dynamical system.

\begin{definition}\label{MDS}
A  \emph{metric dynamical system} (MDS for short)
$\theta=(\Omega,\mathcal{F},\mP,\{\theta_t\}_{t\in\mathbb{R}})$ is a
probability space $(\Omega,\mathcal{F},\mP)$ with a family of measure preserving transformations $\{\theta_t:\Omega\rightarrow\Omega,t\in\mathbb{R}\}$ such that
     \begin{itemize}
     \item[(1)] $\theta_0=\mathrm{id}, \theta_t\circ\theta_s=\theta_{t+s}$ for all $t,s\in\mathbb{R}$;
     \item[(2)] the map $(t,\omega)\mapsto \theta_t\omega$ is measurable and $\theta_t\mP=\mP$ for all $t\in\mathbb{R}$.
     \end{itemize}
 \end{definition}

 \begin{definition}[RDS, \cite{A}]
 A (global) \emph{ random dynamical system} (RDS)  $(\theta, \varphi)$
    on a Polish space $(X,d)$ over an MDS $\theta$ is a mapping
    $$\varphi:\left\{ (s,x,\omega)\in [0,\infty) \times X \times \Omega \right\}\rightarrow X$$
    such that, for each $\omega \in \Omega$,
    \begin{itemize}
     \item[(1)] measurability: $\varphi$ is $(\mathcal{B}([0,\infty))\otimes\mathcal{X}\otimes\mathcal{F},\mathcal{X})$-measurable,
     \item[(2)] $(t,x) \mapsto \varphi_t(x)$ is continuous,
     \item[(3)] $\varphi$ satisfies the following (perfect) {\em cocycle property}:      for all $t,s\geq0$, $x \in X$,
     \begin{align}\label{pcocloc}\varphi_0(.,\omega)=\mathrm{id},\quad \varphi_{t+s}(x,\omega)=\varphi_t(\varphi_s(x,\omega),\theta_s\omega)
     \end{align}
     \end{itemize}
       
 \end{definition}

Clearly, an RDS $\varphi$ induces a flow via $\phi_{s,t}(x):=\varphi_t(x,\theta_s.)$. We say that an SDE generates a flow resp.~an RDS if its solution map has a modification which is a flow resp.~an RDS.
The following study is based on the flow generated by the solution to the SDE with singular drift. Therefore we state  the result from \cite[Theorem 4.5, Corollary 4.10]{LSV} on the  existence of a global semi-flow and a global RDS for singular SDEs under the following condition.
\begin{assumption}\label{Ass}
 For $p,\oops\in(2d,\infty)$  assume
\begin{itemize}
\item[$(i)$] $b\in \tilde L_{p}(\mathbb{R}^d)$, $\sigma:\R^d \to L(\R^d)$ is measurable,  $\|\nabla\sigma\|\in \tilde L_{\oops}(\mathbb{R}^d)$.
\item [$(ii)$] There exist $K_1,K_2>0$ such that for $a:=\sigma\sigma^*$   we have $$K_1|\zeta|^2\leq\langle a(x)\zeta,\zeta\rangle\leq K_2|\zeta|^2,\quad \forall \zeta,x \in\mathbb{R}^d.$$  %$\alpha$-H\"older continuous with
%\begin{align}
 %    \label{a-alpha}
  % \omega_\alpha(a):=\sup_{x,y\in\mathbb{R}^d,x\neq y}\frac{\Vert a(x)-a(y)\Vert}{|x-y|^\alpha}<\infty
 %\end{align}
 %for some $\alpha\in(0,1]$.
\end{itemize}
\end{assumption}
\begin{remark}
   Note that $\tilde L_p \subset \tilde L_{p'}$ whenever $p>p'$. Therefore, if \cref{Ass} holds with different values of $p$ and $\oops$, then it also holds with the larger of the two numbers replaced by 
   the smaller one. In particular, the following result which was formulated for $p=\oops$ can still be applied.
\end{remark}

 \begin{theorem}\cite[Theorem 4.5, Corollary 4.10]{LSV}\label{sSDE}
 If  \cref{Ass} holds,
    then  the SDE \eqref{sde0} admits a flow $\phi$ and a corresponding RDS $\varphi$.
 \end{theorem}
 We will often write $\psi_t(x)$ instead of $\phi_{0,t}(x)$. Abusing notation we will sometimes say "Let $\psi_t(x)$ (or just $\psi$) be a flow ..." instead of "Let $\phi_{s,t}(x), x\in \R^d, 0\leq s \leq t<\infty$ be a flow and $\psi_t(x):=\phi_{0,t}(x)$, $t \geq 0,x \in \R^d$ ...".
 %As the main subject-\emph{ existence of of the random attractor}, 
 %We start with its definition due to \cite{CF}, which usually referred to as \emph{pullback attractors}. 
\begin{definition}[Attractor, \cite{CF}]
Let   $\varphi$ be an  RDS over the MDS $\theta=(\Omega,\mathcal{F},\mP,\{\theta_t\}_{t\in\mathbb{R}})$. The random set $A(\omega)$ is a  \emph{(pullback) attractor} if
\begin{itemize}
    \item[(1)] measurability: $A(\omega)$ is a random element in the metric space of nonempty compact subsets of $X$ equipped with the Hausdorff distance,
    \item[(2)] invariance property: for $t>0$ there exists a set $\Omega_t$ with full measure such that  $$\varphi(t,\omega)(A(\omega))=A(\theta_t\omega),\quad \forall \, \omega\in\Omega_t,$$
    \item[(3)] pull-back limit: almost surely, for all bounded closed sets $B\subset X$,
    $$\lim_{t\rightarrow\infty}\sup_{x\in B}\text{dist}(\varphi(t,\theta_{-t}\omega)(x),A(\omega))=0.$$
\end{itemize}
\end{definition}
One way to verify the existence of an attractor is the following criterion. 
 \begin{proposition}(\cite{CDS}, \cite[Proposition 2.3]{DS})\label{flow-attractor}
 Let   $\varphi$ be an  RDS over the MDS $\theta=(\Omega,\mathcal{F},\mP,\{\theta_t\}_{t\in\mathbb{R}})$. Then the following are equivalent:
 \begin{itemize}
     \item[(i)]  $\varphi$  has an attractor,
     \item[(ii)] $\forall$ $r>0$, $\lim_{R\rightarrow\infty}\mP\Big(\omega \in \Omega: B_r\subset\bigcup_{s=0}^\infty\bigcap_{t\geq s}\varphi^{-1}(t,B_R,\theta_{-t}\omega)\Big)=1.$
 \end{itemize}
 \end{proposition}
\subsection{Main results}
Based on general estimates on the speed of dispersion of random sets in  \cref{section3} (cf. \cref{asy-general})  and on quantitative estimates of the solution to singular SDE in  \cref{section4}, we will show the following result in   \cref{section5}.
\begin{theorem}
If \cref{Ass} holds, then there exists a constant $\kappa>0$ %(see details in \eqref{upper-bound-sde-X}) 
such that for the flow $\psi$ generated by the solution to  \eqref{sde0} we have, for any compact    $\mathcal{X}\subset \mathbb{R}^d$,  
 \begin{align*}
      \limsup_{T\rightarrow\infty}\Big(\sup_{t\in[0,T]}\sup_{x\in\mathcal{X}}\frac{1}{T}|\psi_{t}(x)|\Big)\leq \kappa \quad a.s..
\end{align*}
\end{theorem}
\noindent The precise statement including a formula for $\kappa$ will be given in  \cref{2-point-sde}. There, we can see that $\kappa \to \infty$ as $K_1 \to 0$ (when all other parameters remain unchanged). The following example explains this fact: as the noise becomes more and more degenerate, the linear bound on the dispersion of a bounded set under the flow approaches infinity, so our non-degeneracy assumption on the noise cannot be avoided.
\begin{example}\label{exa:no-linear-growth}
In $\mathbb{R}^2$, for $\epsilon>0$, we consider the  system
\begin{align}\label{eq:system}
    \left\{\begin{array}{rcl}
         \dd X_t&=B(Y_t)\,\dd t+\epsilon\, \dd W_t^1,\quad  X_0\in\mathbb{R},\quad\quad\quad\quad\quad \\
         \dd Y_t&=[((-Y_t)\vee(-1))\wedge1]\,\dd t+\epsilon \,\dd W_t^2,\quad  Y_0\in\mathbb{R},
    \end{array}\right.
\end{align}
where \begin{align*}
   B(y):= \left\{\begin{array}{cc}
        |y|^{-q} &\text{ if } y\neq0,  \\
        0 &  \text{else},
    \end{array}\right.\quad q\in(0,\frac{1}{4}).
\end{align*}
and $W^1, W^2$ are two independent $1$-dimensional Brownian motions.  Notice that for $b(x,y):=(B(y),((-y)\vee(-1))\wedge1)^*$, we have $b\in\tilde L_p(\mathbb{R}^2)$ for $p\in(4,\frac{1}{q})$. Clearly there exists a unique solution $(X,Y)$ to \eqref{eq:system} and
\begin{align*}
    X_t=X_0+\int_0^tB(Y_s)\dd s+W_t^1,\quad t\geq0.
\end{align*}
By the ergodic theorem, almost surely,
\begin{align*}
    \lim_{t\rightarrow\infty}\frac{1}{t}\int_0^tB(Y_s)\,\dd s=\int_{-\infty}^{\infty}B(y)\pi_\epsilon(\dd y),
\end{align*}
where $\pi_\epsilon$ is the invariant probability measure of $Y$. Since $\pi_\epsilon$ converges to the point measure $\delta_0$ weakly as $\epsilon\downarrow0$, we see that the linear expansion rate of $(X,Y)$ converges to $\infty$ when  $\epsilon\downarrow0$. In particular,  we can not expect to have a linear expansion rate for the solution to a singular SDE with degenerate noise in general.
\end{example}

We will now assume that the singular drift $b$ in \eqref{sde0} is of the form  $b=b_1+b_2$ with $b_1\in \tilde L_p(\mathbb{R}^d)$ and $b_2$ satisfies one of the following conditions.
\begin{assumption}\label{Ass-onepoint}
For a given $\beta\in\mR$,  $b_2(x):\mR^d\rightarrow\mR^d$  satisfies
 \begin{itemize}
     \item[($U^\beta$)] $\quad\quad\quad\quad\quad\quad\quad\quad \limsup_{|x|\rightarrow\infty}\frac{x}{|x|}\cdot b_2(x)\leq \beta$
      \end{itemize}
      
     or
     
     \begin{itemize}
      \item[($U_\beta$)] $\quad\quad\quad\quad\quad\quad\quad\quad\liminf_{|x|\rightarrow\infty}\frac{x}{|x|}\cdot b_2(x)\geq \beta$.
 \end{itemize}
\end{assumption}

\begin{theorem}
  Let \cref{Ass} hold. If there exist vector fields $b_1$ and $ b_2$ such that $b=b_1+b_2$ with $b_1\in \tilde L_p(\mR^d)$. There exist positive constants $\beta_1$ (see \cref{non-attractor}) and $\beta_2$  (see \cref{exist-Random-attractor}) such that for the flow $(\psi_t(x))_{t\geq0}$  generated by the solution to  \eqref{sde0}  
  \begin{itemize}
      \item[1.]  if $b_2$ satisfies \cref{Ass-onepoint} $(U_\beta)$ for $\beta>\beta_1$,  then for any $\gamma\in[0,\beta-\beta_1)$ we have
   \begin{align}\label{pro-lowerbound-main-result}
   \lim_{r\rightarrow\infty}\mP\Big(B_{\gamma t}\subset\psi_t(B_r)\quad\forall \quad t\geq0\Big)=1.
   \end{align}
      \item[2.] if $b_2$ satisfies \cref{Ass-onepoint} $(U^\beta)$ for $\beta<-\beta_2,$
   then for any $\gamma\in[0,-\beta-\beta_2)$ we have
   \begin{align}\label{pro-upperbound-main-result}
   \lim_{r\rightarrow\infty}\mP\Big(B_{\gamma t}\subset\psi_{-t,0}^{-1}(B_r)\quad\forall \quad t\geq0\Big)=1.
   \end{align}
   In particular, $\psi$ has a random attractor.
  \end{itemize}
\end{theorem}
Correspondingly the detailed results are presented in \cref{non-attractor} and \cref{exist-Random-attractor}.\\

In the end we give the following example on the special case that the drift is bounded ( i.e. $p=\infty$ ) to conclude the results on the expansion rate and attractors.
\begin{example}
[A case study: bounded coefficients] We consider the flow $(\psi_t(x))_{t\geq0}$ generated by the solution to \eqref{sde0} when $b$, $\nabla \sigma$ are simply bounded, i.e., \cref{Ass} holds with arbitrary $p=\rho\in(1,\infty)$. 
%, for simplicity we take $p=3d$. 
\begin{itemize}
    \item[1.] Expansion rate of the flow: \cref{2-point-sde} shows that for each $\epsilon>0$ there exist constants $C_1$  (depending on $d$ and $\epsilon$) such that for each compact subset $\mathcal{X}\subset\mathbb{R}^d$  
    %(assuming $d\geq2$\footnote{When $d=1$ the result could be more delicate because of the particularity of one dimensional setting.}) only 
    \begin{align}\label{eq:uppbound-bounded-b}
         \limsup_{T\rightarrow\infty}&\Big(\sup_{t\in[0,T]}\sup_{x\in\mathcal{X}}\frac{1}{T}|\psi_t(x)|\Big)\leq C_1\Big(K_2+\Vert b\Vert_{\infty}^2\frac{ K_2}{K_1^2}+{\Vert \nabla \sigma\Vert_{\infty}^2}\Big)\nonumber\\&\quad\quad\quad\quad\quad\quad\quad\quad\quad\quad\quad\Big[\Big(\frac{K_2}{K_1}\Big)^{16d^3+\epsilon}+\Big(\frac{\Vert \nabla \sigma\Vert_{\infty}^2}{K_1}\Big)^{32d^3+\epsilon}+\Big(\frac{\Vert b\Vert_{\tilde L_p}}{K_1}\Big)^{32d^2+\epsilon}\Big]. 
    \end{align}
    \item[2.] Existence of the attractor:   if $b=b_1+b_2$ with $b_1$ bounded and $b_2$ satisfying $(U^\beta)$ in \cref{Ass-onepoint} and $$\beta<-C_2\frac{(\Vert b_1\Vert_{\infty}^2+K_2\Vert b_1\Vert_{\infty})}{{\sqrt{K_1K_2}}}\Big[\Big(\frac{K_2}{K_1}\Big)^{4d^2+\epsilon}+\Big(\frac{\Vert \nabla \sigma\Vert_{\infty}^2}{K_1}\Big)^{4d^2+\epsilon}+\Big(\frac{\Vert b_2\Vert_{\infty}}{K_1}\Big)^{4d+\epsilon}\Big],$$
   where $\epsilon>0$  and  $C_2>0$ is an appropriate function depending on $d$ and $\epsilon$ only, then from \cref{exist-Random-attractor} we know that $\psi$ has an attractor.
\end{itemize}
%It is evident from the above that for singular SDEs the non-degeneracy of the noise is essential. If we try to make the noise vanish, i.e. let $K_1\downarrow 0$ then right hand side of  \eqref{eq:uppbound-bounded-b} will blow up, consequently there is no linear expansion and no attractor which coincides with the fact from \cref{exa:no-linear-growth}.
\end{example}

 \section{Expansion of  sets under a flow}\label{section3}
In this section, we assume that $\psi:[0,\infty)\times \R^d \times \Omega \to \R^d$ is measurable
such that $t \mapsto \psi_t(x,\omega)$ is continuous for every $x \in \R^d$ and $\omega \in \Omega$ (we do not require that $\psi$ has any kind of flow property).
\begin{lemma}\label{2point}
Assume that there exists $\alpha >0$ and a  constant $ c_1>0$  such that for each $r>d$, there exists $c=c(r)>0$  such that for all $x,y\in\mR^d$ and $T>0$, we have
\begin{align}\label{H}
    \Big(\mE\sup_{0\leq t\leq T}(|\psi_t(x)-\psi_t(y)|^r)\Big)^{1/r}\leq c|x-y|\e^{c_1r^\alpha T}.
\end{align}
Then $\psi$ has a modification (which we denote by the same symbol) which is jointly continuous in $(t,x)$ and for each $\gamma>0$ and $u>0$,
\begin{align}\label{Asy}
\limsup_{T\rightarrow\infty}\frac{1}{T}\sup_{\chi_{T,\gamma}}\log\mP\Big(\sup_{x,y\in\chi_{T,\gamma}}\sup_{0\leq t\leq T}|\psi_t(x)-\psi_t(y)|\geq u\Big)\leq -I(\gamma),
\end{align}
where $\sup_{\chi_{T,\gamma}}$ means that we take the supremum over all cubes ${\chi_{T,\gamma}}$ in $\mR^d$ with side length $\e^{-\gamma T}$, and $I:[0,\infty) \to \R$ is defined as
\begin{align}\begin{aligned}
    \label{I}
         I(\gamma):=\left\{\begin{array}{rcl} 
         \gamma^{1+1/\alpha} \alpha (1+\alpha)^{-1-1/\alpha}c_1^{-1/\alpha} \quad\quad\quad \text{     if } &\gamma\geq  c_1(\alpha+1)d^\alpha\\
         d(\gamma-c_1d^\alpha) \quad\quad\quad \text{    if }& c_1d^\alpha<\gamma\leq  c_1(\alpha+1)d^\alpha \\
     0  \quad\quad\quad   \text{     if } &\gamma\leq  c_1d^\alpha. 
    \end{array}\right.
\end{aligned}\end{align}
%Here, $r_{max}:=\frac{\sqrt{3c_2(\gamma-c_0)+c_1^2}-c_1}{3c_2}.$
\end{lemma}
\begin{proof}
We follow the argument in \cite[Proof of Theorem 3.1]{MS}. Without loss of generality we take $\chi:=\chi_{T,\gamma}=[0,\e^{-\gamma T}]^d$ and define $Z_t(x):=\phi_t(\e^{-\gamma T}x)$, $x\in\mR^d$. From \eqref{H} we get
\begin{align*}
    \Big(\mE\sup_{0\leq t\leq T}(|Z_t(x)-Z_t(y)|^r)\Big)^{1/r}\leq c e^{-\gamma T}|x-y|\e^{c_1r^\alpha T}.
\end{align*}
By Kolmogorov's Theorem (see, e.g. \cite[Lemma 2.1]{MS}), $\phi$ admits a jointly continuous modification and for any $\rho\in(0,\frac{r-d}{r})$:
\begin{align}\label{two-point-est-process}
    \mP\Big(\sup_{x,y\in\chi_{T,\gamma}}\sup_{0\leq t\leq T}|\psi_t(x)-\psi_t(y)|\geq u\Big)\leq %\Big(\frac{2d}{1-2^{-\rho}}\Big)^r\frac{c^rd2^{r\rho-(r-d)}}{1-2^{r\rho-(r-d)}}
    \tilde c \e^{(c_1r^\alpha-\gamma)rT}u^{-r},
\end{align}
where $\tilde c$ depends on $r,d,\rho$ only. Taking logarithms, dividing by $T$, then letting $T\rightarrow\infty$ and optimizing over $r>d$ we get the desired result \eqref{Asy}.
\end{proof}
\begin{remark}\label{convex}
Since $I(\gamma)=\sup_{r>d} \big\{r\big(\gamma-c_1r^\alpha\big)\big\}$ is the supremum of affine functions, the map $\gamma \mapsto I(\gamma)$ is convex. Further, $I$ grows faster than linearly.
\end{remark}

The following theorem is a reformulation of \cite[Theorem 2.3]{MS}. 
\begin{theorem}
\label{asy-general}
Let $\psi:[0,\infty)\times \R^d\times \Omega \to \R^d$ be jointly continuous and satisfy the assumptions of  \cref{2point} and \eqref{H} hold with constants $c_1$ and $\alpha$. Assume further, that there exist $c_2 $ and $c_3\geq0$  such that, for each $k>0$ and each bounded set $S\subset\mR^d$, the following holds
    \begin{align}\label{abschae}
        \limsup_{T\rightarrow\infty}\frac{1}{T}\log\sup_{x\in S}\mP\Big(\sup_{0\leq t\leq T}|\psi_t(x)|\geq kT\Big)\leq -c_2k^2+c_3.
    \end{align}
Let $\mathcal{X}$ be a compact subset of $\mR^d$ with box (or upper entropy) dimension $\Delta>0$. Then
\begin{align}
    \label{Upper-bounds-dispersion-sets}
   \limsup_{T\rightarrow\infty}\Big(\sup_{t\in[0,T]}\sup_{x\in\mathcal{X}}\frac{1}{T}|\psi_t(x)|\Big)\leq \kappa \quad a.s.,
\end{align}
where
\begin{equation*}
    \kappa:=\left\{\begin{array}{rcl} \Big(\frac{c_3+\gamma_1\Delta}{c_2}\Big)^{\frac{1}{2}}&\text{ if } \frac{d}{d-\Delta}<\alpha+1,\\
    \Big(\frac{c_3+\gamma_2\Delta}{c_2}\Big)^{\frac{1}{2}}&\text{otherwise},\end{array}\right. \quad\text{ with  } \quad\gamma_1=\frac{c_1d^{\alpha +1}}{d-\Delta},\quad   \gamma_2=c_1(\alpha^{-1}\Delta)^\alpha(1+\alpha)^{1+\alpha}. %\gamma_2=c_1(\lambda^{-1}\Delta)^\lambda(1+\lambda)^{1+\lambda}.
\end{equation*}

\end{theorem}

\begin{remark}\label{oneminusdelta} In addition to the assumptions of the previous theorem, let us assume that $\psi_t(x)=\phi_{0,t}(x)$ where $\phi$ is a flow (later, we  will only consider this case). Let $\mathcal{X}\subset \R^d$ be any compact set and let $B$ be a ball in $\R^d$ containing $\mathcal{X}$. Clearly, the boundary $\partial B$ of $B$ has box dimension $d-1$. The flow property of $\phi$ implies that for each $t\geq 0$, the boundary of $\phi_{0,t}(B)$ is contained in $\phi_{0,t}(\partial B)$ and therefore
any almost sure upper bound $\kappa$ for the linear expansion rate of the set $\partial B$ is at the same time an upper bound for the linear expansion rate of the set $ B$ and hence of $\mathcal{X}$. This means that 
in the case of a flow, the formula for $\kappa$ in the theorem always holds with $\Delta$ replaced by $d-1$ (or the minimum of $\Delta$ and $d-1$).  
%
%Under the additional assumption that for every $t\geq0$, $\phi_t$ is a homeomorphism on $\mathbb{R}^d$,          \cite[Corollary 5.2]{MS}  , and $\mathcal{X}$ is an arbitrary compact subset of $\R^d$, then \eqref{Upper-bounds-dispersion-sets} holds with $\Delta= d-1$. Therefore,  $\frac{d}{d-\Delta}<\frac{c_1(\alpha+1) d^\alpha}{c_1d^\alpha}$ can never hold when $d\geq2$ and $\Delta\leq d-1$. 
\end{remark}

  \section{Quantitative version of Krylov estimates}\label{section4}
 We will show a quantitative version of Krylov estimates \eqref{est-Krlov}.  One can find similar results in the literature with implicit constants, for instance  \cite{KR}, \cite{Zhang2011} and \cite{XXZZ}, which however do not fit our needs since some proofs in later sections rely on the explicit dependence of the constants on the coefficients of the SDE. In the following lemma, a constant $C_{\mathrm {Kry}}$ appears 
 which  depends on $q,p,\oops, d$ only. While we will regard $p,\oops,d$ as fixed throughout, we will apply the formula with different values of $q$ and we will therefore write $C_{\mathrm {Kry}}(q)$ for clarity. 
  \begin{lemma}\label{lemm-Kry}
  If  \cref{Ass} holds and $(X_t)_{t\geq0}$ solves  \eqref{sde0}, then, for $f\in\tilde L_{q}(\mR^d)$ with $q\in (d,\infty]$, there exists a constant $C_{\mathrm{Kry}}(q) >0$
  %$\footnote{This constant appears quite often in later proofs, therefore we give it a unique notation referring to it as the constant from Krylov's estimate.} 
  depending on $q,p,\oops, d$ only such that for $0\leq s\leq t$,
  \begin{align}
      \label{est-Krlov}
      \mE[\int_s^t|f(X_r)|\dd r\Big|\mathcal{F}_s]\leq C_{\mathrm{Kry}}(q)\Gamma\big(K_2^{-\frac{1}{2}}(t-s)^{\frac{1}{2}}+(t-s)\big)\Vert f\Vert_{\tilde L_q},
  \end{align}
  where  $\Gamma:=\big(\frac{K_2}{K_1}\big)^{\frac{4d^2}{1-d/\rho}}+\big(\frac{\Vert\nabla\sigma\Vert_{\tilde L_\oops}^2}{K_1}\big)^{\frac{4d^2}{1-d/\rho}}+\big(\frac{\Vert  b\Vert_{\tilde L_{p}}}{K_1}\big)^{\frac{4d}{1-d/p}}$.
  \end{lemma}
  \begin{proof}
 It is sufficient to show the estimate for positive $f$. \eqref{est-Krlov} clearly holds when $q=\infty$, so we assume $q\in(d,\infty)$. All  positive constants $C_i$, $i=0,\cdots,7$ appearing in the proof only depend on $p,\oops,q, d$. We will regard $p, \oops$ and $d$ as fixed but we will vary $q$ in the following proof and we will therefore highlight the dependence of constants on $q$ in some cases (for $C_0$ and $C_1)$.  First we show that $a:=\sigma\sigma^*$ is $1-\frac{d}{\oops}$-H\"older continuous using Sobolev's embedding theorem and the condition that $ \sigma\in\tilde H^{1,\oops}$ with $\oops>d$. Indeed 
 \begin{align}\label{Holder-a}
   \omega_{1-d/\oops}(a):=&\sup_{x,y\in\mathbb{R}^d,x\neq y,|x-y|\leq 1}\frac{\Vert a(x)-a(y)\Vert}{|x-y|^{1-d/\oops}}  \nonumber\\\leq& \sup_{x,y\in\mathbb{R}^d,x\neq y,|x-y|\leq 1}\big(\frac{\Vert(\sigma\sigma^*)(x)-\sigma(x)\sigma^*(y)\Vert}{|x-y|^{1-d/\oops}}+\frac{\Vert\sigma(x)\sigma^*(y)-(\sigma\sigma^*)(y)\Vert}{|x-y|^{1-d/\oops}}\big)
   \nonumber\\ \leq&\sup_{x,y\in\mathbb{R}^d,x\neq y,|x-y|\leq 1}\big(\frac{\Vert\sigma^*(x)-\sigma^*(y)\Vert\Vert\sigma\Vert_\infty}{|x-y|^{1-d/\oops}}+\frac{\Vert\sigma(x)-\sigma(y)\Vert\Vert\sigma\Vert_\infty}{|x-y|^{1-d/\oops}}\big)\nonumber\\ \leq&  C_{\oops,d} \sqrt{K_2}\Vert\nabla\sigma\Vert_{\tilde L_\oops}.
 \end{align}
 We follow the idea from  \cite[Theorem 3.4]{Zhangzhao2018}.  Applying \cref{aprior-est-local} with $p'=\infty$, we see that there is a unique solution $u\in\tilde{H}^{2,q}$ to
  \begin{align}\label{pde-f}
      \lambda u-\frac{1}{2}a_{ij}\partial_{ij}u=f
  \end{align}
 provided that $\lambda\geq C_0(q)\frac{K_2^2}{K_1}(\frac{K_1+\sqrt{K_2}\Vert\nabla\sigma\Vert_{\tilde L_\oops}}{K_1})^{\frac{2}{1-d/\oops}}=:\lambda_0(q)$. Further, for $\lambda\geq\lambda_0(q)$,  we have
  \begin{align}
      \label{u-nablau}\sup_{x\in\mathbb{R}^d}|u(x)|&\leq C_1(q) \lambda^{-\frac{2-d/q}{2}}{K_1}^{-\frac{d}{2q}}\big(\frac{K_1+\sqrt{K_2}\Vert\nabla\sigma\Vert_{\tilde L_\oops}}{K_1}\big)^{\frac{d}{1-d/\rho}}\Vert f\Vert_{\tilde L_q}=:U_{1,q}(\lambda)\Vert f\Vert_{\tilde L_q},\nonumber\\ \sup_{x\in\mathbb{R}^d}|\nabla u(x)|&\leq C_1(q) \lambda^{-\frac{1-d/q}{2}}K_1^{-\frac{1+d/q}{2}}\big(\frac{K_1+\sqrt{K_2}\Vert\nabla\sigma\Vert_{\tilde L_\oops}}{K_1}\big)^{\frac{d}{1-d/\rho}}\Vert f\Vert_{\tilde L_q}=:U_{2,q}(\lambda)\Vert f\Vert_{\tilde L_q}.
  \end{align}
  Fix $t \geq s \geq 0$ and define the stopping time
  $$\tau_R:=\inf\Big\{\bar s>s:\int_s^{\bar s}\big|b(X_r)\big|\,\dd r\geq R\Big\},\quad 0<R<\infty.$$
  By the generalized It\^o's formula (see e.g. \cite[Lemma 4.1 (iii)]{XXZZ})
  \begin{align*}
      u&(X_{t\wedge\tau_R})-u(X_{s\wedge\tau_R})\\&=\frac{1}{2}\int_{s\wedge\tau_R}^{t\wedge\tau_R}a_{ij}(X_r)\partial_{ij}u(X_r)\,\dd r+\int_{s\wedge\tau_R}^{t\wedge\tau_R}\big(\nabla u(X_r)\big)^* \sigma(X_r)\,\dd W_r
      +\int_{s\wedge\tau_R}^{t\wedge\tau_R}b(X_r)\cdot\nabla u(X_r)\dd r.
  \end{align*}
  Using \eqref{pde-f}, the mean value theorem,   \eqref{u-nablau}  and BDG's inequality, we get that
  \begin{align}\label{Ef}
      \mE\big[&\int_{s\wedge\tau_R}^{t\wedge\tau_R}f(X_r)\dd r \Big|\mathcal{F}_s\big]\nonumber\\=&\mE\big[(u(X_{s\wedge\tau_R})-u(X_{t\wedge\tau_R}))\Big|\mathcal{F}_s\big]+ \mE\big[\lambda\int_{s\wedge\tau_R}^{t\wedge\tau_R}u(X_r)\,\dd r\Big|\mathcal{F}_s\big]+ \mE\big[\int_{s\wedge\tau_R}^{t\wedge\tau_R}b(X_r)\cdot\nabla u(X_r)\,\dd r\Big|\mathcal{F}_s\big]
      \nonumber \\\leq& \sup_{x\in\mathbb{R}^d}|\nabla u(x)|\mE\big[\big|\int_{s\wedge\tau_R}^{t\wedge\tau_R}b(X_r)\,\dd r
      +\int_{s\wedge\tau_R}^{t\wedge\tau_R}\sigma(X_r)\dd W_r\big|\Big|\mathcal{F}_s\big]+\lambda (t-s)\sup_{x\in\mathbb{R}^d}|u(x)|\nonumber \\&\quad \quad+\sup_{x\in\mR^d}|\nabla u(x)|\mE\big[\int_{s\wedge\tau_R}^{t\wedge\tau_R}\big|b(X_r)\big|\dd  r\Big|\mathcal{F}_s\big]
     \nonumber \\\leq& \sup_{x\in\mathbb{R}^d}|\nabla u(x)|C_2\sqrt{K_2}(t-s)^{\frac{1}{2}}+\lambda (t-s)\sup_{x\in\mathbb{R}^d}|u(x)|+2\sup_{x\in\mR^d}|\nabla u(x)|\mE\big[\int_{s\wedge\tau_R}^{t\wedge\tau_R}\big|b(X_r)\big|\dd r\Big|\mathcal{F}_s\big]
      \nonumber \\\leq& C_2\sqrt{K_2}(t-s)^{\frac{1}{2}}U_{2,q}(\lambda)\Vert f\Vert_{\tilde L_q}+\lambda (t-s)U_{1,q}(\lambda)\Vert f\Vert_{\tilde L_q} \nonumber\\&\quad\quad\quad\quad+2U_{2,q}(\lambda)\Vert f\Vert_{\tilde L_q}\mE\big[\int_{s\wedge\tau_R}^{t\wedge\tau_R}\big|b(X_r)\big|\dd r\Big|\mathcal{F}_s\big].
  \end{align}
  Here, the constant $C_2>0$ comes from BDG's inequality.
  We apply this inequality to $f=|b|$ with $q=p$. Then, for $\lambda\geq\lambda_0(p)$,
  \begin{align*}
   \mE\big[\int_{s\wedge\tau_R}^{t\wedge\tau_R}|b(X_r)|\dd r\Big|\mathcal{F}_s\big]   \leq& C_2\sqrt{K_2}(t-s)^{\frac{1}{2}}U_{2,p}(\lambda)\Vert b\Vert_{\tilde L_p}+\lambda (t-s)U_{1,p}(\lambda)\Vert b\Vert_{\tilde L_p}\\&+2U_{2,p}(\lambda)\Vert b\Vert_{\tilde L_p}\mE\big[\int_{s\wedge\tau_R}^{t\wedge\tau_R}|b(X_r)|\dd r\Big|\mathcal{F}_s\big].
  \end{align*}
 If $\lambda \geq \lambda_0(p)$ is so large that $U_{2,p}(\lambda)\Vert b\Vert_{\tilde L_p}=C_1(p)\lambda^{-\frac{1-d/p}{2}}K_1^{-\frac{1+d/p}{2}}\big(\frac{K_1+\sqrt{K_2}\Vert\nabla\sigma\Vert_{\tilde L_\oops}}{K_1}\big)^{\frac{d}{1-d/\rho}}\Vert b\Vert_{\tilde L_p}\leq \frac{1}{4},$ i.e. %>C_0\frac{K_2^2}{K_1}(1+\frac{\sqrt{K_2}\Vert\nabla\sigma\Vert_{\tilde L_p}}{K_1})^{\frac{2d}{1-d/p}}$ and 
 \begin{align}
     \label{U2pb}  \lambda\geq \big( 4C_1(p)
     {K_1}^{\frac{-1-d/p}{2}}\big(\frac{K_1+\sqrt{K_2}\Vert\nabla\sigma\Vert_{\tilde L_\oops}}{K_1}\big)^{\frac{d}{1-d/\rho}}\Vert  b\Vert_{\tilde L_{p}}\big)^{\frac{2}{1-d/p}},
 \end{align} then we get
  \begin{align*}
     \mE\big[\int_{s\wedge\tau_R}^{t\wedge\tau_R}|b(X_r)|\dd r\Big|\mathcal{F}_s\big]   &\leq  \frac{C_2}2\sqrt{K_2}(t-s)^{\frac{1}{2}}+2\lambda (t-s)U_{1,p}(\lambda)\Vert b\Vert_{\tilde L_p}.
  \end{align*}
  Plugging this into \eqref{Ef}, observing that, by definition, $U_{1,p}(\lambda)U_{2,q}(\lambda)=U_{1,q}(\lambda)U_{2,p}(\lambda)$, and using \eqref{U2pb} yields, for $\lambda\geq\lambda_0(p)\vee \lambda_0(q)$ 
  satisfying \eqref{U2pb},
  \begin{align*}
   \mE&\big[\int_{s\wedge\tau_R}^{t\wedge\tau_R}f(X_r)\,\dd r\Big|\mathcal{F}_s\big]   \\\leq &C_3\big(\sqrt{K_2}(t-s)^{\frac{1}{2}}U_{2,q}(\lambda)+\lambda (t-s)(U_{1,q}(\lambda)+U_{1,p}(\lambda)U_{2,q}(\lambda)\Vert b\Vert_{\tilde L_p}\big)\Vert f\Vert_{\tilde L_q}
   %\\\leq &C_6(\sqrt{K_2}(t-s)^{\frac{1}{2}}U_{2,q}+\lambda (t-s))(U_{1,q}+\frac{1}{2}U_{1,q})\Vert f\Vert_{\tilde L_q} 
   \\\leq &2C_3\big(\sqrt{K_2}(t-s)^{\frac{1}{2}}U_{2,q}(\lambda)+\lambda (t-s)U_{1,q}(\lambda)\big)\Vert f\Vert_{\tilde L_q}.
  \end{align*}
Let $\lambda=C_4\big(\frac{K_2^2}{K_1}(\frac{K_1+\sqrt{K_2}\Vert\nabla\sigma\Vert_{\tilde L_\oops}}{K_1})^{\frac{2}{1-d/\oops}}+(4C_1(p){K_1}^{\frac{-1-d/p}{2}}\big(\frac{K_1+\sqrt{K_2}\Vert\nabla\sigma\Vert_{\tilde L_\oops}}{K_1}\big)^{\frac{d}{1-d/\rho}}\Vert  b\Vert_{\tilde L_{p}})^{\frac{2}{1-d/p}}\big)$ with $C_4>C_0(p)\vee C_0(q)\vee 1$, 
which implies 
\begin{align*}
   \sqrt{K_2}U_{2,q}(\lambda)&= C_1(q)\sqrt{K_2}(\lambda K_1)^{-\frac{1}{2}}(\lambda K_1^{-1})^{\frac{d}{2q}}\big(\frac{K_1+\sqrt{K_2}\Vert\nabla\sigma\Vert_{\tilde L_\oops}}{K_1}\big)^{\frac{d}{1-d/\rho}}
   \\&\leq C_5K_2^{-\frac{1}{2}}(\lambda K_1^{-1})^{\frac{d}{2q}}\big(\frac{K_1+\sqrt{K_2}\Vert\nabla\sigma\Vert_{\tilde L_\oops}}{K_1}\big)^{\frac{d}{1-d/\rho}}
   \\&\leq C_6K_2^{-\frac{1}{2}}\Big(\big(\frac{K_2}{K_1}\big)^{\frac{4d^2}{1-d/\rho}}+\big(\frac{\Vert\nabla\sigma\Vert_{\tilde L_\oops}^2}{K_1}\big)^{\frac{4d^2}{1-d/\rho}}+\big(\frac{\Vert  b\Vert_{\tilde L_{p}}}{K_1}\big)^{\frac{4d}{1-d/p}}\Big)
\end{align*}
and 
\begin{align*}
   \lambda U_{1,q}(\lambda)&= C_{1}(q)(\lambda K_1^{-1})^{\frac{d}{2q}}\big(\frac{K_1+\sqrt{K_2}\Vert\nabla\sigma\Vert_{\tilde L_\oops}}{K_1}\big)^{\frac{d}{1-d/\rho}}\\&\leq C_{7}\Big(\big(\frac{K_2}{K_1}\big)^{\frac{4d^2}{1-d/\rho}}+\big(\frac{\Vert\nabla\sigma\Vert_{\tilde L_\oops}^2}{K_1}\big)^{\frac{4d^2}{1-d/\rho}}+\big(\frac{\Vert  b\Vert_{\tilde L_{p}}}{K_1}\big)^{\frac{4d}{1-d/p}}\Big).
\end{align*}
In the above estimates we used the fact that $p>2d$ and $q>d$. Therefore,
  \begin{align}\label{R1}
   \mE&\Big[\int_{s\wedge\tau_R}^{t\wedge\tau_R}f(X_r)\dd r\Big|\mathcal{F}_s\Big]
    \nonumber\\&\leq   C_{\mathrm{Kry}}(q)\Big(\big(\frac{K_2}{K_1}\big)^{\frac{4d^2}{1-d/\rho}}+\big(\frac{\Vert\nabla\sigma\Vert_{\tilde L_\oops}^2}{K_1}\big)^{\frac{4d^2}{1-d/\rho}}+\big(\frac{\Vert  b\Vert_{\tilde L_{p}}}{K_1}\big)^{\frac{4d}{1-d/p}}\Big)
   [K_2^{-\frac{1}{2}}(t-s)^{\frac{1}{2}}+(t-s)]\Vert f\Vert_{\tilde L_q}.
  \end{align}
    Letting $R\rightarrow \infty$ we therefore get \eqref{est-Krlov}.  
    \end{proof}
The following corollary is a quantitative version of  Khasminskii's lemma. The constant $C_{\mathrm{Kry}}(q)$ appearing in there is the same as in the previous lemma.

\begin{corollary}\label{est-exp}
Let \cref{Ass} hold, let $\Gamma:=\Big(\big(\frac{K_2}{K_1}\big)^{\frac{4d^2}{1-d/\rho}}+\big(\frac{\Vert\nabla\sigma\Vert_{\tilde L_\oops}^2}{K_1}\big)^{\frac{4d^2}{1-d/\rho}}+\big(\frac{\Vert  b\Vert_{\tilde L_{p}}}{K_1}\big)^{\frac{4d}{1-d/p}}\Big)$. Then, for any $f\in\tilde L_q(\mR^d)$ with $q\in (d,\infty]$, any $0\leq S \leq T$, and any $0<\lambda<\infty$, the solution $(X_t)_{t\geq0}$ of \eqref{sde0} satisfies
\begin{align}
    \label{Khas-Kry-gener1}
    \mE\exp\Big(\lambda\int_S^T|f(X_r)|\dd r\Big)\leq  
    %2\cdot 2^{C_{\mathrm{Kry}}(T-S)\lambda\Gamma\Vert f\Vert_{\tilde L_q}\big(\sqrt{1+C_{\mathrm{Kry}}\lambda\Gamma\Vert f\Vert_{\tilde L_q}K_2^{-1}}-\sqrt{C_{\mathrm{Kry}}\lambda\Gamma\Vert f\Vert_{\tilde L_q}K_2^{-1}}\big)^{-2}}
     2\cdot 2^{(T-S)\big(\frac{\kappa}{2}K_2^{-1/2}+\sqrt{\frac{\kappa^2}{4}K_2^{-1}+\kappa}   \big)^2}
    \leq 2\cdot 2^{(T-S)\big(\frac{\kappa^2}{K_2}+2\kappa\big)},
\end{align}
where $\kappa:=2C_{\mathrm{Kry}}(q)\lambda\Gamma\Vert f\Vert_{\tilde L_q}$.
\end{corollary}
\begin{proof} The second inequality is an application of the general inequality $(A+B)^2\leq 2A^2 + 2B^2$.

\cref{lemm-Kry} shows that there exists some positive integer $n$ such that,  for $j=0,\cdots,n-1$,
\begin{align}\label{half}
    \lambda \mE\Big[\int_{\frac{(T-S)j}{n}}^{\frac{(T-S)(j+1)}{n}}\Big|f(X_r)\Big|\dd  r\Big|\mathcal{F}_{\frac{(T-S)j}{n}}\Big] \leq \frac{1}{2}
\end{align}
and the proof of \cite[Lemma 3.5]{XieZhang2017} shows that for any such $n$ we have  
\begin{align*}
     \mE\exp\Big(\lambda\int_S^T|f(X_r)|\dd r\Big)\leq 2^n
\end{align*}
(see also \cite[Lemma 3.5]{LL}). 
By \cref{lemm-Kry}, any  $n$ such that
\begin{align*}
    C_{\mathrm{Kry}}(q)\lambda \Gamma \Vert f\Vert_{\tilde L_q}\Big[\Big(\frac{T-S}{K_2n}\Big)^{\frac{1}{2}}+\frac{T-S}{n}\Big]\leq \frac{1}{2} 
\end{align*}
satisfies \eqref{half}. In particular, we can take
%$C_0>0$ is from \eqref{est-Krlov}.  
$$
n=\Big\lfloor (T-S) \Big( \frac{\kappa}{2}K_2^{-1/2}+\sqrt{\frac{\kappa^2}{4}K_2^{-1}+\kappa}   \Big)^2     \Big\rfloor+1
$$
%$$ n=\lfloor(T-S)\frac{4C_{\mathrm{Kry}}\lambda\Gamma\Vert f\Vert_{\tilde L_q}K_2}{\big(\sqrt{2K_2+C_{\mathrm{Kry}}\lambda\Gamma\Vert f\Vert_{\tilde L_q}}-\sqrt{C_{\mathrm{Kry}}\lambda\Gamma\Vert f\Vert_{\tilde L_q}}\big)^2}\rfloor+1.$$
Here,  $\lfloor x\rfloor$ is the largest integer   that is smaller than or equal to $x \in \R$. Therefore \eqref{Khas-Kry-gener1} holds.
\end{proof}
 
  \begin{remark}
  Note that the right hand side of our version of Krylov's estimate contains the factor   $(t-s)^{1/2}+(t-s)$ 
  instead of $C(T)(t-s)^{1-\frac{d}{2q}}$  in \cite[Theorem 3.4 (3.8)]{Zhangzhao2018}), where $C(T)$ depends on the final time $T$. Further, we require the condition $q>d$ instead of $q>d/2$ in   \cite[Theorem 3.4 (3.8)]{Zhangzhao2018}). The reason for our restriction to $q>d$ is that we use \eqref{u-nablau} which only holds for $q>d$. Since we will later apply Krylov's estimate to $f:=|b^*\cdot\sigma^{-1}|^2$ which is in 
  $\tilde L_{p/2}$ we will have to assume $p>2d$.
   \end{remark}
  
  \begin{remark}
   More general versions  of the quantitative Khasminskii's Lemma (but with less explicit constants) can be found in \cite{Le}. 
  \end{remark}
  
 \section{Upper bounds for the dispersion of sets induced by the flow generated by the solution to   SDE}\label{section5}
 Depending on the regularity of the SDE's coefficients   we show  upper bounds for the dispersion of sets under the flow generated by the solution  in the following two cases.
 \subsection{Stability estimates of the  SDE with weakly differentiable coefficients}
 Consider the equation
\begin{align}\label{SDEY}\dd  Y_t^i=\tilde b(Y_t^i)\,\dd t+\tilde\sigma(Y_t^i)\,\dd W_t, \quad Y_0^i=y_i\in\mathbb{R}^d,\quad  i=1,2.
  \end{align}
For $\tilde b$ and $\tilde \sigma$ we assume:
\begin{assumption}\label{ass-transformed}
For $p,\rho\in(2d,\infty)$, 
 \begin{itemize}
     \item[1.] $\Vert  \tilde  b\Vert_{\tilde H^{1,p}}+\Vert \tilde b\Vert_{\infty}<\infty$;
     \item[2.] $\Vert \nabla\tilde \sigma\Vert_{\tilde L_\rho}<\infty$;
     \item[3.] for $\tilde a:=\tilde \sigma\tilde \sigma^*$, there exist some $\tilde K_1,\tilde K_2>0$ such that for all $x\in\mathbb{R}^d$, $$\tilde K_1|\zeta|^2\leq\langle\tilde a(x)\zeta,\zeta\rangle\leq \tilde K_2|\zeta|^2,\quad \forall \zeta\in\mathbb{R}^d.  $$  %$\tilde\alpha$-H\"older continuous with
%\begin{align}
 %    \label{a-alpha}
  % \omega_{\tilde\alpha}(\tilde a):=\sup_{x,y\in\mathbb{R}^d,x\neq y}\frac{\Vert \tilde a(x)-\tilde a(y)\Vert}{|x-y|^{\tilde \alpha}}<\infty
 %\end{align}
 %for some $\tilde\alpha\in(0,1]$.
 \end{itemize}
\end{assumption}

\begin{theorem}\label{Transformed-two-point}
Let \cref{ass-transformed} hold.  There exist constants   $\kappa_0,\kappa_1>0$ depending only on $p,d$, $\rho$, such that for any $r\geq 1$, $T \ge 0$, $y_i\in\mR^d$, $i=1,2,$  the solutions $Y^i:=Y^i(y_i)$ to equations \eqref{SDEY} satisfy
\begin{align}\label{est-Y-sobolev}
\mE[\sup_{t\in[0,T]}|Y_t^1(y_1)-Y_t^2(y_2)|^r]
\leq \kappa_0|y_1-y_2|^r\exp(\kappa_1T\varrho),
\end{align}
where
\begin{align}\label{rho-zero-Y-sobolev}
    \varrho:= r^4\Big[\Vert \tilde b\Vert_{\infty}+\Vert \tilde\sigma\Vert_\infty^2
+(\tilde\Gamma\Vert \nabla\tilde b\Vert_{\tilde L_p})^2\tilde K_2^{-1}+\tilde\Gamma\Vert \nabla\tilde b\Vert_{\tilde L_p} +\tilde\Gamma^2\Vert \nabla\tilde\sigma\Vert_{\tilde L_\oops}^4\tilde K_2^{-1}+\tilde\Gamma\Vert \nabla\tilde\sigma\Vert_{\tilde L_\oops}^2\Big].
\end{align}
and   $\tilde \Gamma:=\Big(\big(\frac{\tilde K_2}{\tilde K_1}\big)^{\frac{4d^2}{1-d/\rho}}+\big(\frac{\Vert\nabla\tilde\sigma\Vert_{\tilde L_\oops}^2}{\tilde K_1}\big)^{\frac{4d^2}{1-d/\rho}}+\big(\frac{\Vert  \tilde b\Vert_{\tilde L_{p}}}{\tilde K_1}\big)^{\frac{4d}{1-d/p}}\Big)$.
\end{theorem}
\begin{proof} Again, all constants $C_1,...$ depend on $p,\rho,d$ only.
By It\^o's formula we get for any $r\geq1$,
\begin{align}\label{pre-Gron}
    |Y_t^1-Y_t^2|^{2r}=&|y_1-y_2|^{2r}+\int_0^t|Y_s^1-Y_s^2|^{2r}\dd A_s+M_t\leq |y_1-y_2|^{2r}+\int_0^t|Y_s^1-Y_s^2|^{2r}\dd \bar A_s+M_t,
\end{align}
where $M_t$ is an $(\mathcal{F}_t)$-local martingale defined as
$$M_t:=\int_0^t2r|Y_s^1-Y_s^2|^{2r-2}[\tilde \sigma(Y_s^1)-\tilde \sigma(Y_s^2)]^*(Y_s^1-Y_s^2)\,\dd W_s$$
and
\begin{align*}
    A_t:=&\int_0^t\frac{2r\langle Y_s^1-Y_s^2,\tilde b(Y_s^1)-\tilde b(Y_s^2)\rangle+r\Vert\tilde \sigma(Y_s^1)-\tilde \sigma(Y_s^2)\Vert^2}{|Y_s^1-Y_s^2|^2}\dd s
    \\&+\int_0^t\frac{2r(r-1)|[\tilde \sigma(Y_s^1)-\tilde \sigma(Y_s^2)]^*(Y_s^1-Y_s^2)|^2}{|Y_s^1-Y_s^2|^4}\dd s
\end{align*}
and
\begin{align*}
    \bar A_t:=&\int_0^t\frac{2r|\langle Y_s^1-Y_s^2,\tilde b(Y_s^1)-\tilde b(Y_s^2)\rangle|+r\Vert\tilde \sigma(Y_s^1)-\tilde \sigma(Y_s^2)\Vert^2}{|Y_s^1-Y_s^2|^2}\dd s
    \\&+\int_0^t\frac{2r(r-1)|[\tilde \sigma(Y_s^1)-\tilde \sigma(Y_s^2)]^*(Y_s^1-Y_s^2)|^2}{|Y_s^1-Y_s^2|^4}\dd s.
\end{align*}

There exists $C_1>0$   such that for each $x,y \in \R^d$
\begin{align*}
    |\tilde \sigma(x)-\tilde \sigma(y)|&\leq C_1|x-y|(\mathcal{M}|\nabla\tilde\sigma|(x)+\mathcal{M}|\nabla\tilde\sigma|(y)+\Vert\tilde\sigma\Vert_\infty),\\
    |\tilde b(x)-\tilde b(y)|&\leq C_1|x-y|(\mathcal{M}|\nabla\tilde b|(x)+\mathcal{M}|\nabla\tilde b|(y)+\Vert\tilde b\Vert_\infty),
\end{align*}
where $\mathcal{M}f$ is defined as $\mathcal{M} f(x):=\sup_{r\in (0,1)}\frac{1}{|B_r|}\int_{B_r}f(x+y)\dd y,$
which satisfies  
\begin{align}\label{max-ineq}\Vert \mathcal{M}f\Vert_{\tilde L_\gamma}\leq C(\gamma,d)\Vert f\Vert_{\tilde L_\gamma}\quad \text{ for }\quad \gamma>1,\end{align}
see  \cite[Lemma 2.1]{XXZZ}.

Using these estimates and the Cauchy–Schwarz inequality, we get 
\begin{align*}
    \bar A_t\leq &C_2\Big(r\Big(\int_0^t\mathcal{M}|\nabla \tilde b|(Y_s^1)+\mathcal{M}|\nabla \tilde b|(Y_s^2)\dd s+t\Vert\tilde b\Vert_\infty\Big)
    \\&+r\Big(\int_0^t\mathcal{M}|\nabla \tilde \sigma|^2(Y_s^1)+\mathcal{M}|\nabla \tilde \sigma|^2(Y_s^2)\dd s+t\Vert\tilde \sigma\Vert_\infty^2\Big)
    \\
    &+2r(r-1)\Big(\int_0^t\mathcal{M}|\nabla \tilde \sigma|^2(Y_s^1)+\mathcal{M}|\nabla \tilde \sigma|^2(Y_s^2)\dd s+t\Vert\tilde \sigma\Vert_\infty^2\Big)\Big)
    \\= & tC_2 \big(r\Vert\tilde b\Vert_\infty+(2r^2-r)\Vert\tilde \sigma\Vert_\infty^2\big)
    \\&+ C_2   \sum_{i=1}^2 \int_0^t  r\mathcal{M}|\nabla \tilde b|(Y_s^i)+(2r^2-r)\mathcal{M}|\nabla \tilde \sigma|^2(Y_s^i)\,\dd s.
%    \\&+C_3\Big(\int_0^t(2r\mathcal{M}|\nabla \tilde b|+(2r^2-r)\mathcal{M}|\nabla \tilde \sigma|^2)(Y_s^1)\dd s
%    \\&\quad\quad\quad+\int_0^t(2r\mathcal{M}|\nabla \tilde b|+(2r^2-r)\mathcal{M}|\nabla \tilde \sigma|^2)(Y_s^2)\dd s\Big)
%    \\=&:tA^0+\int_0^t(2rA_{0,s}^{1,1}+(2r^2-r)A_{0,s}^{1,2}+2rA_{0,s}^{2,1}+(2r^2-r)A_{0,s}^{2,2})\dd s.
\end{align*}
Applying  \cref{est-exp} and \eqref{max-ineq} we get,  for $\alpha>0$ and $t\geq 0$, 
%\begin{align*}
%        \mE\Big[\exp\Big(\int_0^t2r\alpha A_{0,s}^{i,1}\dd s\Big)\Big]&\leq 2\cdot 2^{C_3t\big((r\alpha\tilde\Gamma\Vert \nabla \tilde b\Vert_{\tilde L_p})^2\tilde K_2^{-1}+r\alpha\tilde\Gamma\Vert \nabla \tilde b\Vert_{\tilde L_p}\big)}
%     \\&\leq \exp\Big(C_4t\big((r\alpha\tilde\Gamma\Vert \nabla \tilde b\Vert_{\tilde L_p})^2\tilde K_2^{-1}+r\alpha\tilde\Gamma\Vert \nabla \tilde b\Vert_{\tilde L_p}\big)\Big) \\&=:\exp(\rho_{i,1}),\\
%           \mE\Big[\exp\Big(\int_0^t\alpha (2r^2-r)A_{0,s}^{i,2}\dd s\Big)\Big]&\leq 2\cdot 2^{C_5t\big((\alpha r^2\tilde\Gamma^2\Vert \nabla\tilde\sigma\Vert_{\tilde L_\rho}^2)^2\tilde K_2^{-1}+(\alpha r^2\tilde\Gamma^2\Vert \nabla\tilde\sigma\Vert_{\tilde L_\rho}^2)\big)}
%          \\&\leq \exp\Big(C_6t\big((\alpha r^2\tilde\Gamma^2\Vert \nabla\tilde\sigma\Vert_{\tilde L_\rho}^2)^2\tilde K_2^{-1}+(\alpha r^2\tilde\Gamma^2\Vert \nabla\tilde\sigma\Vert_{\tilde L_\rho}^2)\big)\Big) 
%          \\&=:\exp(\rho_{i,2}).
%\end{align*}
 \begin{align}\label{alphaA}
    \mE[\exp(\alpha \bar A_t)]\leq 16 \exp\big[C_3\varrho_\alpha t\big],
%    \exp\Big[C_7t\Big(\alpha(2r\Vert\tilde b\Vert_\infty+(2r^2-r)\Vert\tilde \sigma\Vert_\infty^2)&+\sum_{i,j=1}^2\rho_{i,j}\Big)\Big]=:\exp{[C_7t\rho_\alpha]}.
\end{align}
where
%We can find there exists $C_8>0$ such that
\begin{align}\label{rho-est}
    \varrho_\alpha=  &\alpha \big(r\Vert\tilde b\Vert_\infty+r^2\Vert\tilde \sigma\Vert_\infty^2\big)
    +  (r\alpha\tilde\Gamma\Vert \nabla \tilde b\Vert_{\tilde L_p})^2\tilde K_2^{-1}+r\alpha\tilde\Gamma\Vert \nabla \tilde b\Vert_{\tilde L_p} 
   \nonumber\\&+(\alpha r^2\tilde\Gamma\Vert \nabla\tilde\sigma\Vert_{\tilde L_\rho}^2)^2\tilde K_2^{-1}+(\alpha r^2\tilde\Gamma\Vert \nabla\tilde\sigma\Vert_{\tilde L_\rho}^2).
\end{align}
Choosing $\alpha=1$ and
applying stochastic Gr\"onwall's inequality (see \cite[Theorem 4]{MS13} or \cite[Lemma 3.7]{XieZhang2017}) to \eqref{pre-Gron} we get %for any $\kappa\in(\frac{1}{2},1)$ \rho_1T
$$
    \mE[\sup_{t\in[0,T]}|Y_t^1-Y_t^2|^r]
    \leq C_4  |y_1-y_2|^r \Big(\mE \big[\exp \big(\bar A_T\big)\big] \Big)^{1/2}\leq 4 C_4  |y_1-y_2|^r
     \exp \Big(\frac12 C_3 \varrho_1 T\Big).
 %   &\leq \Big(\frac{\kappa}{\kappa-\frac{1}{2}}\Big)^{\frac{1}{2}}|y_1-y_2|^r\Big(\mE\exp(\frac{\kappa}{1-\kappa}A_T)\Big)^{\frac{1-\kappa}{\kappa}}
  %  \\&\leq \Big(\frac{\kappa}{\kappa-\frac{1}{2}}\Big)^{\frac{1}{2}}|y_1-y_2|^r\exp\Big(C_9T\frac{1-\kappa}{\kappa}\rho_{\frac{\kappa}{1-\kappa}}\Big).
$$
%Let $\kappa\in(\frac{1}{2},1)$ be some constant and  $\alpha={\frac{\kappa}{1-\kappa}}$, from \eqref{alphaA}, \eqref{rho-est},
%\eqref{b-sigma}, \eqref{tilde-b} and \eqref{tilde-sigma},
%take $\lambda=C_0(K_2+((1+\frac{1}{\sqrt{K_1}})\Vert b\Vert_{\tilde {{L}_{p}}})^{\frac{2p}{p-d}})$ and
Observing that $\varrho_1$ is at most equal to $\varrho_0$ defined in \eqref{rho-zero-Y-sobolev} and defining $\kappa_0=4C_4$ and $\kappa_1=\frac 12 C_3$, \eqref{est-Y-sobolev} follows.
%
%\begin{align*}
 %   &\mE[\sup_{t\in[0,T]}|Y_t^1-Y_t^2|^r]
%\\&\leq  C_{10}|y_1-y_2|^r\exp\Big( C_{11}Tr^4\Big[\Vert \tilde b\Vert_{\infty}+\Vert \tilde\sigma\Vert_\infty^2
%+(\tilde\Gamma\Vert \nabla \tilde b\Vert_{\tilde L_p})^2\tilde K_2^{-1}+(\tilde\Gamma\Vert \nabla \tilde b\Vert_{\tilde L_p})\\& \quad\quad\quad\quad\quad\quad\quad\quad\quad\quad\quad\quad\quad\quad\quad\quad\quad+(\tilde\Gamma\Vert \nabla\tilde\sigma\Vert_{\tilde L_\rho})^4\tilde K_2^{-1}+(\tilde\Gamma\Vert \nabla\tilde\sigma\Vert_{\tilde L_\rho})^2\Big]\Big).
%\end{align*}
%It shows that \eqref{est-Y-sobolev} holds with $\rho_0$ as defined in \eqref{rho-zero-Y-sobolev}.
\end{proof}

\begin{remark}
 If $\tilde \sigma$ is even globally Lipschitz continuous with Lipschitz constant $L$, then there is no need to use 
 Khasminskii's Lemma for the integral over $\tilde \sigma$ and we easily get \eqref{est-Y-sobolev} with
 $$
 \varrho= r^2  
 \Big[\Vert \tilde b\Vert_{\infty}
+(\tilde\Gamma\Vert \nabla\tilde b\Vert_{\tilde L_p})^2\tilde K_2^{-1}+\tilde\Gamma\Vert \nabla\tilde b\Vert_{\tilde L_p}+L^2\Big]
 $$
 and
 $$
\tilde \Gamma:=\Big(\big(\frac{\tilde K_2}{\tilde K_1}\big)^{4d^2}+\big(\frac{L}{\tilde K_1}\big)^{4d^2}+\big(\frac{\Vert  \tilde b\Vert_{\tilde L_{p}}}{\tilde K_1}\big)^{\frac{4d}{1-d/p}}\Big).
 $$
 \end{remark}

\subsection{Linear expansion rate of the  SDE with singular coefficients}
%In this section, we assume that the dimension $d$ of the state space is at least 2.

\begin{theorem}\label{2-point-sde}
Let \cref{Ass} hold.  Let $(\psi_t)_{t\geq0}$ denote the flow generated by the solution to  \eqref{sde0}. Let $\mathcal{X}$ be a compact subset of   $\mathbb{R}^d$. Then there exists a  positive constant $C_{p,\rho,d}$ depending on $p, d, \rho$ only  such that 
 \begin{align}
     \label{upper-bound-sde-X}
      \limsup_{T\rightarrow\infty}\Big(\sup_{t\in[0,T]}\sup_{x\in\mathcal{X}}\frac{1}{T}|\psi_t(x)|\Big)\leq \kappa^* \quad a.s.,
\end{align}
where 
\begin{align*} 
 \kappa^* = & C_{p,\rho,d}\Big(K_2+\Vert b\Vert_{\tilde L_p}^2 \frac{ K_2}{K_1^2}+{\Vert \nabla \sigma\Vert_{\tilde L_\oops}^2}\Big)\nonumber\\&
 \quad\quad\Big[\Big(\frac{K_2}{K_1}\Big)^{\frac{16d^3}{(1-d/(p\wedge\rho))(1-d/\rho)}}
     +\Big(\frac{\Vert b\Vert_{\tilde L_p}}{K_1}\Big)^{\frac{32d^2}{1-d/(p\wedge\rho)}}
  +\Big(\frac{\Vert\nabla\sigma\Vert^2_{L_\oops}}{K_1}\Big)^{\frac{32d^3}{(1-d/(p\wedge\rho))(1-d/\rho)}}\Big].
\end{align*}
\iffalse
$\kappa$ is taken as follows:
\begin{equation*}
    \kappa:=\left\{\begin{array}{rcl} c_3+\Big(\frac{\gamma_1(d-1)}{c_2}\Big)^{\frac{1}{2}}&\text{ if } d<\frac{c_0+2c_1d}{c_0+c_1d},\\
    c_3+\Big(\frac{\gamma_2(d-1)}{c_2}\Big)^{\frac{1}{2}}&\text{otherwise}.\end{array}\right. \quad\quad  %\gamma_2=c_1(\lambda^{-1}\Delta)^\lambda(1+\lambda)^{1+\lambda}.
\end{equation*}
with $\gamma_1=d(c_0+c_1d),\quad\gamma_2=\sqrt{(2c_1(d-1)+c_0)^2-c_0^2}+(2c_1(d-1)+c_0),$
and
\begin{align*}
\lambda:&= C_1^*K_1\Big(\frac{K_2^2}{K_1^2}(\frac{K_1+\sqrt{K_2}\Vert\nabla\sigma\Vert_{\tilde L_p}}{K_1})^{\frac{2d}{1-d/p}}+(\frac{\Vert b\Vert_{\tilde L_p}}{K_1})^{\frac{2}{1-d/p}}\Big),\\
\Gamma^*:&=\frac{K_2^2}{K_1^2}(\frac{K_1^2+K_2\Vert b\Vert_{\tilde L_p}+K_1\sqrt{K_2}\Vert \nabla\sigma\Vert_{\tilde L_p}}{K_1^2})^\frac{d^2}{p-d}+(\frac{\lambda}{K_1})^{\frac{d}{p-d}},\\
\Gamma:&=\frac{K_2^2}{K_1^2}(\frac{K_1+\sqrt{K_2}\Vert\nabla\sigma\Vert_{\tilde L_p}}{K_1})^{\frac{d^2}{p-d}}+\Big(\frac{\Vert b\Vert_{\tilde L_p}}{K_1}\Big)^{\frac{d}{p-d}},
\\
 c_0:&=C_3^*\Big(2\lambda+2\lambda\Gamma^*\Big),\quad\quad c_1:=C_3^*\Big(K_2+{\Gamma^*}^2(C_2^*\frac{\sqrt{K_2}}{K_1}\Vert b\Vert_{\tilde L_p}+\Vert \nabla\sigma\Vert_{\tilde L_p})^2\Big),\\
 c_2:&=\frac 1{4C_3^* K_2  },
\quad \quad c_3:=C_3^* \Gamma\Vert  b\Vert_{\tilde L_p}.
\end{align*}
\fi
\end{theorem}
\begin{proof}
The idea is to apply \cref{asy-general}. All constants $C_1^*,...$ depend on $p,\oops,d$ only.\\ 

\noindent Step 1. We check the assumptions of Lemma \ref{2point}. \\
%Verification of condition (i) in  \cref{asy-general}.\\
    
Since, by \eqref{Holder-a}, the map $x \mapsto a(x)=\sigma (x)\sigma^*(x)$  is $1-d/\rho$-H\"older continuous  and $\omega_{1-d/\rho}(a)\leq C_{\rho,d}\sqrt{K_2}\Vert\nabla\sigma\Vert_{\tilde L_\rho}$,  \cref{aprior-est-local}  and \cref{homo} 
show that there exists a constant $C_1^*$  such that for 
$$\lambda:=C_1^*K_1\Big(\frac{K_2^2}{K_1^2}\big(\frac{K_1+\sqrt{K_2}\Vert\nabla\sigma\Vert_{\tilde L_\rho}}{K_1}\big)^{\frac{2}{1-d/\rho}}+\big(\frac{K_1+\sqrt{K_2}\Vert\nabla\sigma\Vert_{\tilde L_\oops}}{K_1}\big)^{\frac{2d}{(1-d/\rho)(1-d/p)}}\big(\frac{\Vert  b\Vert_{\tilde L_{p}}}{K_1}\big)^{\frac{2}{1-d/p}}\Big),
$$ the equation 
%Let $U:=(u^{(l)})_{1\leq l\leq d}$ and $u^{(l)}\in\tilde H^{2,p}$ be the unique solutions to the equations
$$
\frac{1}{2}a_{ij}\partial_{ij}^2u^{(l)}+b\cdot\nabla u^{(l)}-\lambda u^{(l)}=-b^{(l)}, \quad l=1,\cdots,d,
$$
has a unique solution $U:=(u^{(l)})_{1\leq l\leq d}$, $u^{(l)}\in\tilde H^{2,p}$ and 
\begin{align}\label{pdephi}
\Phi(x):=x+U(x)\quad \text{ for }x\in\mathbb{R}^{d} 
\end{align} 
%Let $$\lambda:=C_1^*K_1\Big(\frac{K_2^2}{K_1^2}\big(\frac{K_1+\sqrt{K_2}\Vert\nabla\sigma\Vert_{\tilde L_p}}{K_1}\big)^{\frac{2d}{1-d/p}}+\big(\frac{\Vert  b\Vert_{\tilde L_{p}}}{K_1}\big)^{\frac{2}{1-d/p}}\Big),$$  %be large chosen later such that $\lambda>C_0K_1(1+\frac{\sqrt{K_2}\Vert\nabla\sigma\Vert_{\tilde L_p}}{K_1})^{\frac{2d}{1-d/p}}$ and $C_0 \lambda^{-\frac{1-d/p}{2}}K_1^{-\frac{1+d/p}{2}}\Vert b\Vert_{\tilde L_p}\leq \frac{1}{2}$ 
%where $C_1^*>0$ is some positive constant depending only on $p,d$.
is  a $C^1$-diffeomorphism on $\mathbb{R}^d$ (see also \cite{Zhangzhao2018}).  Let 
 $ \Psi:=(\Phi)^{-1}.$
Then, by the generalized It\^o's formula (\cite{XXZZ}), $Y_t:=\Phi(\psi_t(x))$ satisfies the following equation
 \begin{align}\label{SDEY1}\dd Y_t=\tilde b(Y_t)\,\dd t+\tilde\sigma(Y_t)\,\dd  W_t, \quad Y_0=y\in\mathbb{R}^d
  \end{align}
  with
  $$\tilde b(x):=\lambda U(\Psi(x)),\quad\tilde\sigma(x):=[\nabla\Phi\cdot\sigma]\circ({\Psi}(x)),\quad y=\Phi(x).$$
%Then
%we  take $C_1^*>1\vee C_0$ and
%$$\lambda:=C_1^*K_1(\frac{K_2^2}{K_1^2}(\frac{K_1+\sqrt{K_2}\Vert\nabla\sigma\Vert_{\tilde L_p}}{K_1})^{\frac{2d}{1-d/p}}+(\frac{\Vert  b\Vert_{\tilde L_{p}}}{K_1})^{\frac{2}{1-d/p}}).$$
From \cite[(4.5)]{XXZZ} we know that 
\begin{align}\label{u-infty}
\Vert U\Vert_\infty<\frac{1}{2},\quad \Vert \nabla U\Vert_\infty<\frac{1}{2}.    
\end{align}
Furthermore, by \eqref{gradu-infty} and \eqref{aprior-estimate-pde} we have 
\begin{align}\label{b-sigma}
 \Vert\nabla U\Vert_{\tilde L_{p}}\leq \frac{1}{2}\big(\frac{K_1}{\lambda}\big)^\frac{d}{2p}\leq\frac{1}{2},\quad \Vert U\Vert_{\tilde L_{p}}\leq \frac{1}{2}\big(\frac{K_1}{\lambda}\big)^\frac{1-d/p}{2}\leq \frac{1}{2},\nonumber\\\Vert\nabla^2 U\Vert_{\tilde L_p}\leq C_2^*\frac{1}{K_1}\big(1+\frac{\sqrt{K_2}\Vert\nabla\sigma\Vert_{\tilde L_\oops}}{K_1}\big)^{\frac{d}{(1-d/\rho)}}\Vert b\Vert_{\tilde L_p}.
\end{align}
Hence, by \eqref{u-infty} (see also e.g. \cite[p. 15]{XXZZ}),
$$\frac{1}{2}\leq|\nabla\Phi|=|\mathbb{I}+\nabla U|\leq \frac{3}{2},\quad |\nabla \Psi|\leq 2$$
which implies that for all $x\in\mathbb{R}^d$,
\begin{align}
    \label{k-sigma} \frac{1}{4} K_1|\xi|^2\leq\langle\tilde \sigma\tilde\sigma^*(x)\xi,\xi\rangle\leq \frac{9}{4}K_2|\xi|^2,\quad \forall \xi\in\mathbb{R}^d,
\end{align}
and
\begin{align}\label{tilde-b}
    \Vert\tilde b\Vert_\infty&\leq \lambda\Vert  U\Vert_{\infty}\leq \frac{1}{2}\lambda, \quad
    \Vert\tilde b\Vert_{\tilde L_p}\leq \lambda\Vert  U\Vert_{\tilde L_p} \leq \frac{1}{2}\lambda,\quad
   \nonumber\\ \Vert\nabla\tilde b\Vert_{\tilde L_p}&\leq \lambda \Vert \text{det}(\nabla\Phi)\Vert_\infty^{\frac{1}{p}}\Vert \nabla U\Vert_{\tilde L_p}\leq \lambda.
    %C(K_2+\frac{K_2(1+\sqrt{K_2}+K_2)}{C-(1+\sqrt{K_2}+K_2)} +(\frac{1}{\sqrt{K_1}}\frac{1}{\sqrt{K_2}}\Vert b\Vert_{\tilde L_p}+1)^{\frac{2}{1-d/p}}),
\end{align}
Moreover for $p'=\min(p,\rho)$ we have by embedding
\begin{align}\label{tilde-sigma}
      \Vert\nabla\tilde \sigma\Vert_{\tilde L_{p'}}&=\Vert\Big((\nabla^2\Phi\cdot \sigma+\nabla\Phi\nabla\sigma)\nabla\Psi\Big)\circ\Psi\Vert_{\tilde L_{p'}}
      \nonumber\\&\leq \Vert\Big((\nabla^2\Phi\cdot \sigma)\nabla\Psi\Big)\circ\Psi\Vert_{\tilde L_{p}}+\Vert\Big((\nabla\Phi\nabla\sigma)\nabla\Psi\Big)\circ\Psi\Vert_{\tilde L_{\rho}}
      \nonumber\\&\leq 2\Vert \text{det}(\nabla\Phi)\Vert_\infty^{\frac{1}{p\wedge\rho}}(\sqrt{K_2}\Vert\nabla^2\Phi\Vert_{\tilde L_p}+\Vert\nabla\Phi\cdot\nabla\Psi\Vert_{\infty}\Vert\nabla\sigma\Vert_{\tilde L_\rho})
     \nonumber \\&\leq 9C_2^*\frac{\sqrt{K_2}}{K_1}\big(1+\frac{\sqrt{K_2}\Vert\nabla\sigma\Vert_{\tilde L_\oops}}{K_1}\big)^{\frac{d}{(1-d/\rho)}}\Vert b\Vert_{\tilde L_p}+9\Vert \nabla\sigma\Vert_{\tilde L_\rho}.
\end{align}
If $(\phi_t(x))_{t\geq0}$ is the flow generated by the solution to  \eqref{SDEY1}, then by definition of $\Phi(\psi_t(x))$ from \eqref{pdephi} and the fact that $U$ is uniformly bounded from \eqref{u-infty}, we get that
\begin{align*}
   \limsup_{T\rightarrow\infty}\Big(\sup_{t\in[0,T]}\sup_{x\in\mathcal{X}}\frac{1}{T}|\psi_t(x)|\Big)=\limsup_{T\rightarrow\infty}\Big(\sup_{t\in[0,T]}\sup_{x\in\mathcal{X}}\frac{1}{T}|\phi_t(x)|\Big).
\end{align*}
Using  the estimates \eqref{k-sigma}, \eqref{tilde-b} and \eqref{tilde-sigma} we will establish \eqref{est-Y-sobolev} for $Y$. Indeed, let  $\tilde K_1:=\frac{1}{4}K_1$ and $\tilde K_2=\frac{9}{4}K_2$ in \cref{ass-transformed}.
\iffalse Notice that by Sobolev embedding theorem $\nabla\Phi^b$ and $\sigma$ are $1-\frac{d}{p}$-H\"older continuous, so
\begin{align*}
   &\sup_{x,y\in\mathbb{R}^d,x\neq y}\frac{\Vert\tilde\sigma\tilde\sigma^*(x)-\tilde\sigma\tilde\sigma^*(y)\Vert}{|x-y|^{1-p/d}}\\\leq& \Vert(\nabla\Phi^b)^*\nabla\Phi^b\Vert_\infty\sup_{x,y\in\mathbb{R}^d,x\neq y}\frac{\Vert(\sigma\sigma^*)\circ \Psi^b(x)-(\sigma\sigma^*)\circ \Psi^b(y)\Vert}{|x-y|^{1-p/d}}
    \\&+\Vert\sigma^*\sigma\Vert_\infty\sup_{x,y\in\mathbb{R}^d,x\neq y}\frac{\Vert(\nabla\Phi^b\nabla\Phi^b)^*\circ\Psi^b(x)-(\nabla\Phi^b\nabla\Phi^b)^*\circ\Psi^b(y)\Vert}{|x-y|^{1-p/d}}
    \\\leq&\frac{\sqrt{K_2}}{4}\Vert\nabla\sigma\Vert_{\tilde L_p}+K_2\Vert\nabla^2
    U_b\Vert_{\tilde L_p}
    \\\leq & \frac{\sqrt{K_2}}{4}\Vert\nabla\sigma\Vert_{\tilde L_p}+C_2^*K_2\frac{\Vert b\Vert_{\tilde L_p}}{K_1}=: \omega_{1-p/d}(\tilde a). 
\end{align*}
Inside the above estimates we applied the inequalities 
\begin{align*}
  \Vert(\sigma\sigma^*)\circ \Psi^b(x)-\sigma\sigma^*\circ \Psi^b(y)\Vert\leq & \Vert(\sigma\sigma^*)\circ \Psi^b(x)-\sigma\circ \Psi^b(y)\sigma^*\circ \Psi^b(x)\Vert\\&+ \Vert\sigma\circ \Psi^b(y)\sigma^*\circ \Psi^b(x)-(\sigma\sigma^*)\circ \Psi^b(y)\Vert
  \\\leq &\Vert\sigma^*\Vert_\infty\Vert\sigma\circ \Psi^b(x)-\sigma\circ \Psi^b(y)\Vert
  \\&+ \Vert\sigma\Vert_\infty\Vert\sigma^*\circ \Psi^b(x)-\sigma^*\circ \Psi^b(y)\Vert.
\end{align*}
The same idea applied to the estimates of $\nabla\Phi^b\nabla\Phi^b)^*\circ\Psi^b$.\fi 
Then  we define
\begin{align}\label{tildeGamma}
   \tilde \Gamma:=&\Big(\big(\frac{\tilde K_2}{\tilde K_1}\big)^{\frac{4d^2}{1-d/p'}}+\big(\frac{\Vert\tilde \nabla\sigma\Vert_{\tilde L_{p'}}^2}{\tilde K_1}\big)^{\frac{4d^2}{1-d/p'}}+\big(\frac{\Vert  \tilde b\Vert_{\tilde L_{p}}}{\tilde K_1}\big)^{\frac{4d}{1-d/p}}\Big)
  \nonumber \\\leq& C_{p,\rho,d}\Big(\big(\frac{K_2}{K_1}\big)^{\frac{4d^2}{1-d/(p\wedge\rho)}}+\big(\frac{K_2}{K_1}\frac{\Vert b\Vert_{\tilde L_p}^2}{K_1^2}(1+\frac{\sqrt{K_2}\Vert\nabla\sigma\Vert_{\tilde L_\rho}}{K_1})^{\frac{2d}{1-d/\rho}}\big)^{\frac{4d^2}{1-d/(p\wedge\rho)}}+(\frac{\lambda}{K_1})^{\frac{4d}{p-d}}\Big)
  \nonumber\\\leq& C_{p,\rho,d}\Big(\big(\frac{K_2}{K_1}\big)^{\frac{8d^3}{(1-d/(p\wedge\rho))(1-d/\rho)}}+\big(\frac{\Vert b\Vert_{\tilde L_p}}{K_1}\big)^{\frac{16d^2}{1-d/(p\wedge\rho)}}
 +\big(\frac{\Vert\nabla\sigma\Vert^2}{K_1}\big)^{\frac{16d^3}{(1-d/(p\wedge\rho))(1-d/\rho)}}\Big). 
\end{align}
Using \cref{Transformed-two-point} and the fact that  $|\nabla\Psi|\leq 2$ together with \eqref{k-sigma}, \eqref{tilde-b} and \eqref{tilde-sigma},  for the flows correspondingly $\psi_t^1(x_1),\psi_t^2(x_2)$ generated by the solutions $X_t^1(x_1)$, $X_t^1(x_2)$ to \eqref{sde0} we get
%that there exist positive constants $C_3^*$ depending on $p,\rho$ and $d$ and  $C_4^*$ depending on $r$   such that 
\begin{align}\label{est:exp-X}
  \mE[\sup_{t\in[0,T]}|\psi_t^1(x_1)-\psi_t^2(x_2)|^r]
  &= \mE[\sup_{t\in[0,T]}|\Psi(Y_t^1(y_1))-\Psi(Y_t^2(y_2))|^r]
\nonumber\\&\leq 2^r\mE[\sup_{t\in[0,T]}|Y_t^1(y_1)-Y_t^2(y_2)|^r]\leq 2^r C_4^*|y_1-y_2|^r\exp(C_3^*T\varrho)
\end{align}
with
\begin{align}\label{rho-zero-X-sobolev}
    \varrho:=&r^4\Big [ \lambda+K_2
+ \tilde\Gamma\lambda+(\tilde\Gamma\lambda)^2K_2^{-1} + \tilde\Gamma^2(\frac{K_2}{K_1^2}\Vert b\Vert_{\tilde L_p}^2+\Vert \nabla\sigma\Vert_{\tilde L_\rho}^2)^2K_2^{-1}+\tilde\Gamma(\frac{K_2}{K_1^2}\Vert b\Vert_{\tilde L_p}^2+\Vert \nabla\sigma\Vert_{\tilde L_\rho}^2)\Big].
\end{align}  
\begin{itemize}
    \item [Step 2.] Verification of estimate  \eqref{abschae} in  \cref{asy-general}.
\end{itemize}

Let $$\rho_t:=\exp\Big(\int_0^tb^*(\sigma^{-1})^*(\varphi_r(x))\dd W_r-\frac{1}{2}\int_0^tb^*(\sigma\sigma^*)^{-1}b(\varphi_r(x))\dd r\Big),$$
where $\varphi_t(x)$ is the flow generated by the solution to
\begin{align*}
   \dd \varphi_t=\sigma(\varphi_t)\dd W_t,\quad \varphi_0(x)=x\in\mR^d.
\end{align*}
It follows from \eqref{Khas-Kry-gener1} that, for  any $\beta>0$, 
\begin{align}\label{est:girsanov-novikov}
    \mE&\exp\Big(\beta\int_0^Tb^*(\sigma\sigma^*)^{-1}b(\varphi_r(x))\dd r\Big)
 \leq 2\exp\Big(TC_5^*\big((K_1^2K_2)^{-1}{(\Gamma'\beta)^2\Vert b\Vert_{\tilde L_p}^4}+\Gamma'\beta K_1^{-1}\Vert b\Vert_{\tilde L_p}^2\big)\Big)
\end{align}
where 
\begin{align}
    \label{gamma-prime}\Gamma'=\big(\frac{K_2}{K_1}\big)^{\frac{4d^2}{1-d/\rho}}+\big(\frac{\Vert\nabla\sigma\Vert_{\tilde L_\oops}^2}{K_1}\big)^{\frac{4d^2}{1-d/\rho}}.
\end{align} 
%and $C_{\mathrm{Kry}}>0$ is from \eqref{est-Krlov} with $q=p$ and $b=0$.
Therefore $(\rho_t)_{t\geq0}$ is a martingale. Let $\mP^\rho:=\rho_T\mP$. By Girsanov's theorem and H\"older's inequality,
\begin{align*}
\mP\Big(\sup_{0\leq t\leq T}|\psi_t(x)-x|\geq kT\Big)
  =&\mP^\rho\Big(\sup_{0\leq t\leq T}|\varphi_t(x)-x|\geq kT\Big)\\=&\mE[\rho_T\mathbb{I}_{\left\{\sup_{0\leq t\leq T}|\varphi_t(x)-x|\geq kT\right\}}]
\\\leq & [\mE\rho_T^2]^\frac{1}{2}\mP[\sup_{0\leq t\leq T}|\varphi_t(x)-x|\geq kT]^\frac{1}{2}  .
\end{align*}
Applying Markov's inequality we obtain, for each $x \in \R^d$ and $\zeta\geq 0$,
 \begin{align}\label{one-point}
     \mP\Big(\sup_{0\leq t\leq T}|\varphi_t(x)-x|\geq kT\Big)^{1/2}\leq e^{-\frac{1}{2}\zeta kT}\Big[\mE\exp\Big(\zeta \sup_{0\leq t\leq T}\Big|\int_0^t\sigma(\varphi_r(x))\dd W_r\Big|\Big)\Big]^\frac{1}{2}.
 \end{align}
  \eqref{est:girsanov-novikov} shows 
 \begin{align*}
     \Big[\mE\rho_T^2\Big]^{1/2}&=\Big[ \mE\exp\Big(2\int_0^Tb^*(\sigma^{-1})^*(\varphi_r(x))\dd W_r-2\int_0^Tb^*(\sigma\sigma^*)^{-1}b(\varphi_r(x))\dd r\nonumber
   \\&\quad\quad \quad\quad +\int_0^Tb^*(\sigma\sigma^*)^{-1}b(\varphi_r(x))\dd r\Big)\Big]^{1/2}
     \\&\leq\Big(\mE\Big[\exp\Big(2\int_0^Tb^*(\sigma^{-1})^*(\varphi_r(x))\dd W_r-2\int_0^Tb^*(\sigma\sigma^*)^{-1}b(\varphi_r(x))\dd r\Big]^2\Big)^{1/4}\nonumber\\&\quad \Big[\mE\exp\Big(\int_0^t2b^*(\sigma\sigma^*)^{-1}b(\varphi_r(x))\dd r\Big)\Big]^{1/4}
    \\&\leq \Big[\mE\exp\Big(2\int_0^Tb^*(\sigma\sigma^*)^{-1}b(\varphi_r(x))\dd r\Big)\Big]^{1/4}
 \nonumber \\&\leq   2\exp\Big(C_5^*T\big((K_1^2K_2)^{-1}{\Gamma'^2\Vert b\Vert_{\tilde L_p}^4}+K_1^{-1}\Gamma'\Vert b\Vert_{\tilde L_p}^2\big)\Big)
      =:2\exp(T\kappa_1)
 \end{align*}
 and by time change $\int_0^t\sigma(\varphi_r(r))\dd W_r=W_{\int_0^t|\sigma(\varphi_r(x))|^2\dd r}$, 
 we also have
 \begin{align*}
   \Big[\mE\exp\Big(2\zeta\sup_{0\leq t\leq T}\Big|\int_0^t\sigma(\varphi_r(x))\dd W_r\Big|\Big)\Big]^{1/2}  \leq \sqrt{2} \exp(C_d\zeta^2 \Vert \sigma\Vert_\infty^2 T)=:\sqrt{2}\exp(T\zeta^2 \kappa_2).
 \end{align*}
 Inserting these estimate into \eqref{one-point}  and optimizing over $\zeta\ge 0$ yields, for any $k>0$,
\begin{align}\label{one-point-X}
   \mP\Big(\sup_{0\leq t\leq T}|\psi_t(x)-x|\geq kT\Big) \nonumber &\leq 2\exp\Big(C_6^*T{\big(\kappa_1+\kappa_2\zeta^2 -\zeta k\big)}\Big)
   \nonumber \\ &\leq 2\exp\Big(C_7^*T{\big(-\frac{1}{4\kappa_2}k^2+\kappa_1\big)}\Big).
\end{align}

With estimates \eqref{est:exp-X} and \eqref{one-point-X} at hand we are ready to apply \cref{asy-general} by taking %{\color{magenta}{to do}} 
\begin{align}\label{c0c1c2}
  c_1:&= \lambda+K_2
+\tilde\Gamma\lambda+ (\tilde\Gamma\lambda)^2K_2^{-1} + \tilde\Gamma^2(\frac{K_2}{K_1^2}\Vert b\Vert_{\tilde L_p}^2+\Vert \nabla\sigma\Vert_{\tilde L_\rho}^2)^2K_2^{-1}+ \tilde\Gamma(\frac{K_2}{K_1^2}\Vert b\Vert_{\tilde L_p}^2+\Vert \nabla\sigma\Vert_{\tilde L_\rho}^2),\nonumber\\
 c_2:&=\frac{1}{4\Vert\sigma\Vert_\infty^2},
\quad \quad
c_3:=C_7^*(K_1^2K_2)^{-1}{\Gamma'}^2{\Vert b\Vert_{\tilde L_p}^4}+K_1^{-1}\Gamma'\Vert b\Vert_{\tilde L_p}^2,\quad \alpha:= 3,
\end{align}
with $\tilde \Gamma$ from \eqref{tildeGamma} and $\Gamma'$ from \eqref{gamma-prime}.
 Note that we can take $\Delta=d-1$ by  Remark \ref{oneminusdelta}.   The linear expansion rate $\kappa$ can now be estimated as follows (no matter which of the two cases in the definition of $\kappa$ in \cref{asy-general} applies): 
 %After some algebraic calculation we obtain the expansion rate \footnote{Notice that in \eqref{Upper-bounds-dispersion-sets} there are two cases for $\kappa$ for $d<\alpha+1=4$ and $d\geq\alpha+1$. Considering we do not distinguish the constant $C_{\alpha,d}$ for different $d$, only one $\kappa$ is necessary to be stated here.}
 \begin{align}\label{kapparoh}
     \kappa&\leq C_{\alpha,d} \Big(\frac{c_1+c_3}{c_2}\Big)^{1/2}\nonumber
     \\&\leq C_{p,\rho,d} \Vert\sigma\Vert_\infty\Big(\sqrt{\lambda}+\sqrt{K_2}+\sqrt{\tilde\Gamma\lambda}+\tilde\Gamma \lambda K_2^{-1/2}+\tilde{\Gamma}\big(K_2K_1^{-2}\Vert b\Vert_{\tilde L_p}^2+\Vert \nabla\sigma\Vert_{\tilde L_\rho}^2\big)K_2^{-1/2}\nonumber
     \\&\quad\quad\quad\quad\quad\quad\quad+\sqrt{\tilde{\Gamma}}\big(\sqrt{K_2}K_1^{-1}\Vert b\Vert_{\tilde L_p}+\Vert \nabla\sigma\Vert_{\tilde L_\rho}\big)+(K_1^2K_2)^{-1/2}\Gamma'{\Vert b\Vert_{\tilde L_p}^2}+(K_1^{-1}\Gamma')^{1/2}\Vert b\Vert_{\tilde L_p}\Big)
     \nonumber
     %%%%%%%%%%%%%%%%%%%%%%%%%%
     \\&\leq C_{p,\rho,d} \sqrt{K_2}\Big(\sqrt{K_1}+\sqrt{K_2}+\frac{K_1}{\sqrt{K_2}}+\Vert b\Vert_{\tilde L_p}^2(\frac{\sqrt{K_2}}{K_1^2}+\frac{1}{K_1\sqrt{K_2}})+\frac{\sqrt{K_2}+\sqrt{K_1}}{K_1}\Vert b\Vert_{\tilde L_p}+\Vert \nabla\sigma\Vert_{\tilde L_\rho}
     %%%%%%%%%%%%%%%%%%%%%%%%%%%%%%%%%%%%5
     \nonumber\\&\quad\quad\quad\quad\quad+\frac{\Vert \nabla \sigma\Vert_{\tilde L_\rho}^2}{\sqrt{K_2}}\Big)\Big[\Big(\frac{K_2}{K_1}\Big)^{\frac{16d^3}{(1-d/(p\wedge\rho))(1-d/\rho)}}+\big(\frac{\Vert b\Vert_{\tilde L_p}}{K_1}\big)^{\frac{32d^2}{1-d/(p\wedge\rho)}}+\big(\frac{\Vert b\Vert_{\tilde L_p}}{K_1}\big)^{\frac{8}{1-d/p}}
  \nonumber\\&\quad\quad\quad\quad\quad+\big(\frac{\Vert\nabla\sigma\Vert_{\tilde L_\rho}^2}{K_1}\big)^{\frac{32d^3}{(1-d/(p\wedge\rho))(1-d/\rho)}}+(\frac{\Vert\nabla\sigma\Vert_{\tilde L_\rho}^2}{K_1}\big)^{\frac{8d}{(1-d/p)(1-d/\rho)}}+\big(\frac{\Vert\nabla\sigma\Vert_{\tilde L_\rho}^2}{K_1}\big)^{\frac{8}{1-d/\rho}}\Big] \nonumber
  %%%%%%%%%%%%%%%%%%%%%%%%%%%%%%%%%%%%%%%%
     \\&\leq C_{p,\rho,d} \Big(K_2+\Vert b\Vert_{\tilde L_p}^2 \frac{ K_2}{K_1^2}+{\Vert \nabla \sigma\Vert_{\tilde L_\oops}^2}\Big)\nonumber
     \\&\quad\quad\quad\quad\quad\Big[\Big(\frac{K_2}{K_1}\Big)^{\frac{16d^3}{(1-d/(p\wedge\rho))(1-d/\rho)}}
     +\Big(\frac{\Vert b\Vert_{\tilde L_p}}{K_1}\Big)^{\frac{32d^2}{1-d/(p\wedge\rho)}}
  +\Big(\frac{\Vert\nabla\sigma\Vert^2_{L_\oops}}{K_1}\Big)^{\frac{32d^3}{(1-d/(p\wedge\rho))(1-d/\rho)}}\Big].
 \end{align}
In the last inequality we used  that %$\min(\frac{8d^2}{p(1-d/(p\wedge\rho))},\frac{4}{1-d/\rho}\frac{d}{p-d},\frac{4}{1-d/\rho})=\frac{4}{1-d/\rho}\frac{d}{p-d}$, 
$\max(\frac{32d^2}{1-d/(p\wedge\rho)},\frac{8}{1-d/p})\leq \frac{32d^2}{1-d/(p\wedge\rho)},$
%$\min({\frac{8d}{1-d/\rho}\frac{d}{p-d}},{\frac{4d^2}{p(1-d/\rho)}},{\frac{8d^2}{p(1-d/(p\wedge\rho))}},\frac{8d}{1-d/\rho})\geq {\frac{2d}{(1-d/\rho)}\frac{d}{p-d}}$, 
 and $\max({\frac{32d^3}{(1-d/(p\wedge\rho))(1-d/\rho)}},$ ${\frac{8d}{(1-d/p)(1-d/\rho)}},\frac{8}{1-d/\rho})\leq {\frac{32d^3}{(1-d/(p\wedge\rho))(1-d/\rho)}}$.  In the end we get \eqref{upper-bound-sde-X}.
\end{proof}
As a by-product from the proof of \cref{2-point-sde} we also have
\begin{proposition}\label{two-point-prop}
Let $(\psi_t(x))_{t\geq0}$ denote the flow generated by the solution to  \eqref{sde0}.  Let $\chi_T$ be cubes of   $\mathbb{R}^d$  with side length $\exp(-\gamma T)$, $\gamma>0$. If \cref{Ass} holds then  for any $k>0$
 \begin{align*}
     \limsup_{T\rightarrow\infty}\frac{1}{T}\sup_{\chi_T}\log\mP\Big(\sup_{x,y\in \chi_T}\sup_{0\leq t\leq T}|\psi_t(x)-\psi_t(y)|\geq k\Big) \leq -I(\gamma)
 \end{align*}
 where
 \begin{align}\begin{aligned}
    \label{I-X}
     I(\gamma):=\left\{\begin{array}{rcl} 
         \gamma^{1+1/\alpha} \alpha (1+\alpha)^{-1-1/\alpha}c_1^{-1/\alpha} \quad\quad\quad \text{     if } &\gamma\geq  c_1(\alpha+1)d^\alpha\\
         d(\gamma-c_1d^\alpha) \quad\quad\quad \text{    if }& c_1d^\alpha<\gamma\leq  c_1(\alpha+1)d^\alpha \\
     0  \quad\quad\quad   \text{     if } &\gamma\leq  c_1d^\alpha. 
    \end{array}\right.
\end{aligned}\end{align}
with $\alpha$ and $c_1$  as in \eqref{c0c1c2}.
\end{proposition}
\begin{proof}
This  follows easily from \eqref{rho-zero-X-sobolev} and \cref{2point}.
\end{proof}
 \section{Existence of random attractors to SDEs with singular drift}\label{section6}
 Inspired by the work \cite{DS}, we are interested in the question whether there exists a random attractor of the RDS generated by the solution to the singular SDE. We start with estimates of the one-point motion (items 1-5 of the following lemma) and then move to estimates for the dispersion of sets (items 6 and 7). 
 \begin{lemma}\label{lemm-mul}
Let \cref{Ass} hold. Further assume that there exist vector fields $b_1$ and $ b_2$ such that $b=b_1+b_2$ with $b_1\in \tilde L_p(\mR^d)$. Let $(\psi_t(x))_{t\geq0}$ be the flow generated by the solution to  \eqref{sde0}. Let $\Gamma:=C_{\mathrm{Kry}}(\frac{p}{2})\Big(\big(\frac{K_2}{K_1}\big)^{\frac{4d^2}{1-d/\rho}}+\big(\frac{\Vert\nabla\sigma\Vert_{\tilde L_\oops}^2}{K_1}\big)^{\frac{4d^2}{1-d/\rho}}+\big(\frac{\Vert b_2\Vert_{\tilde L_{p}}}{K_1}\big)^{\frac{4d}{1-d/p}}\Big)$  where  $C_{\mathrm{Kry}}(\frac{p}{2})$  is  from \eqref{est-Krlov} with $q=\frac{p}{2}$ depending on $p, \rho$ and $d$ only.  
\begin{itemize}
    \item[1.] Let $1\leq r$, and $r_1,r_2>r$. If $b_2$ satisfies \cref{Ass-onepoint} $(U^\beta)$ for some $\beta \in \R$, then, for each $|x|=r_2$,
    \begin{align*}
        \mP\Big(&|\psi_T(x)|\geq r_1,\inf_{0\leq t\leq T}|\psi_t(x)|\geq r\Big)
        \\&\leq  2 \exp\Big(T\frac{\Gamma^2\Vert b_1\Vert_{\tilde L_p}^4+K_2^2\Gamma\Vert b_1\Vert_{\tilde L_p}^2}{{K_1K_2^2}}-\frac{1}{4}\Big(-\frac{r_2-r_1}{\sqrt{K_2T}}-\frac{\sqrt{T}\beta^*(r)}{\sqrt{K_2}}\Big)_+^2\Big)
    \end{align*}
    with 
   \begin{align}
       \label{beta-up}
       \beta^*(r):=\sup_{|x|\geq r}\frac{x\cdot b_2(x)}{|x|}+(d-1)\frac{K_2}{2r}.
   \end{align}
 \item[2.] If  $b_2$ satisfies \cref{Ass-onepoint} $(U^\beta)$ for some  $\beta<0$ and  $r_0>1$ is such that $\beta^*(r_0)\leq 0$ where $\beta^*(r_0)$ is from \eqref{beta-up}, then for every $R\geq r\geq r_0$ and every $x\in\mR^d$, we have
\begin{align*}
    \mP\Big(|\psi_T(x)|\geq R,\inf_{0\leq t\leq T}|\psi_t(x)|\leq r\Big)\leq 4\exp\Big(T\frac{\Gamma^2\Vert b_1\Vert_{\tilde L_p}^4+K_2^2\Gamma\Vert b_1\Vert_{\tilde L_p}^2}{{K_1K_2^2}}-\frac{(R-r)^2}{16K_2T}\Big).
\end{align*}
\item[3.]If  $b_2$ satisfies \cref{Ass-onepoint} $(U^\beta)$ for some  $\beta<0$ and  $r_0>1$ such that $\beta^*(r_0)\leq 0$ where $\beta^*(r_0)$ is from \eqref{beta-up} and if $R\geq r_0$, then for every $|x|=R$, $\delta,\delta_1>0$, we have 
\begin{align*}
    \mP\Big(\sup_{0\leq s\leq \delta_1}|\psi_s(x)|\geq R+\delta\Big)\leq 6\exp\Big(T\frac{\Gamma^2\Vert b_1\Vert_{\tilde L_p}^4+K_2^2\Gamma\Vert b_1\Vert_{\tilde L_p}^2}{{K_1K_2^2}}-\frac{\delta^2}{16K_2\delta_1}\Big).
\end{align*}
    \item[4.] Let $1\leq r$, and $r_1,r_2>r$. If $b_2$ satisfies \cref{Ass-onepoint} $(U_\beta)$ for some $\beta \in \R$, then for each $|x|=r_1$, 
     \begin{align*}
         \mP\Big(|\psi_T(x)|\leq r_2,&\inf_{0\leq t\leq T}|\psi_t(x)|\geq r\Big)\\&\leq  2 \exp\Big(T\frac{\Gamma^2\Vert b_1\Vert_{\tilde L_p}^4+K_2^2\Gamma\Vert b_1\Vert_{\tilde L_p}^2}{{K_1K_2^2}}-\frac{1}{4}\Big(\frac{\sqrt{T}\beta_*(r)}{\sqrt{K_2}}-\frac{r_2-r_1}{\sqrt{K_2T}}\Big)_+^2\Big)
    \end{align*}
with
 \begin{align}
       \label{beta-down}
       \beta_*(r):=\inf_{|x|\geq r}\frac{x\cdot b_2(x)}{|x|}.
   \end{align}
\item[5.]  If $b_2$ satisfies \cref{Ass-onepoint} $(U_\beta)$ for some $\beta \in \R$, then for each $|x|=r_1$, for $1\leq r<r_1$
\begin{align*}
       \mP&\Big(\inf_{t\geq 0}|\psi_t(x)|\leq r\Big)\leq 2\exp\Big(T\frac{\Gamma^2\Vert b_1\Vert_{\tilde L_p}^4+K_2^2\Gamma\Vert b_1\Vert_{\tilde L_p}^2}{{K_1K_2^2}}-(r_1-r)\frac{\beta_*(r)}{K_2}\Big)
\end{align*}
with $\beta_*(r_1)$ defined as \eqref{beta-down}.
\item[6.] Assume that  $b_2$ satisfies \cref{Ass-onepoint} $(U_\beta)$ for $$\beta>\beta_0:= 4\frac{\Vert b_1\Vert_{\tilde L_p}^2\Gamma+K_2\Vert b_1\Vert_{\tilde L_p}\sqrt\Gamma}{{\sqrt{K_1K_2}}}.$$ Let $h:[1,\infty)\rightarrow[1,\infty)$ be strictly increasing such that $\lim_{x\rightarrow\infty}\frac{h(x)}{x}=0$ and $\lim_{x\rightarrow\infty}\frac{\log x}{h(x)}=0$. Let $\eta\in(0,\frac{1}{2})$ and $\gamma>0$ with $\eta+\gamma<\beta-\beta_0$. For $R>2$, define $T:=h(R)$, $r=(1-\eta)R$ and $r_1:=R+\gamma h(R)$. Then
\begin{align*}
    \limsup_{R\rightarrow\infty}\frac{1}{h(R)}\log\mP_{R}
    &:=\limsup_{R\rightarrow\infty}\frac{1}{h(R)}
    \log\mP\Big[\Big(B_{r_1}\nsubseteq\psi_T(B_{R})\Big)\cup\cup_{t\in[0,T]}\Big(B_{r}\nsubseteq\psi_t(B_{R})\Big)\Big]<0.
\end{align*}
\item[7.] Assume that $b_2$ satisfies \cref{Ass-onepoint} $(U^\beta)$ for $$\beta<-\beta_0:= -4\frac{\Vert b_1\Vert_{\tilde L_p}^2\Gamma+K_2\Vert b_1\Vert_{\tilde L_p}\sqrt\Gamma}{{\sqrt{K_1K_2}}}.$$ Let $h(R)=R^\iota$ for some $\iota\in(0,\frac{1}{3})$.  Let $\eta\in(0,\frac{1}{2})$ and $\gamma>0$ with $\eta+\gamma<-\beta-\beta_0$. For $R>2$, define $T:=h(R)$, $r=(1-\eta)R$ and $r_1:=R+\gamma h(R)$. Then
\begin{align*}
    \limsup_{R\rightarrow\infty}&\frac{1}{h(R)}\log\mP_{R}
    \\&:=\limsup_{R\rightarrow\infty}\frac{1}{h(R)}
    \log\mP\Big[\bigcup_{|x|=r_1}\Big((|\psi_T(x)|\geq R)\cap(\inf_{t\in[0,T]}|\psi_{t}(x)|\geq r)\Big)\Big]<0.
\end{align*}
\end{itemize}
 \end{lemma}
 \begin{proof}
Let us explain the idea of the proof of parts 1 to 5: we express the probabilities on the left side by the corresponding ones for the flow $\psi^2$ generated by the SDE with drift $b$ replaced by $b_2$ by applying Girsanov's theorem. This is possible since $b_1\in \tilde L_p$.   The required estimates for $\psi^2$ 
can then be obtained from results in \cite{DS}. 
Notice that strictly speaking the SDE generating $\psi^2$ cannot be applied since the assumptions in \cite{DS} require the coefficients to be one-sided Lipschitz continuous which is not necessarily true in our set-up. 
It is easy to check however that the estimates of the one-point motion in Propositions 4.2 to 4.6 in \cite{DS} hold without additional Lipschitz-type assumptions. Therefore, we divide the proof into two steps: a Girsanov argument and  then estimates for the flow $\psi^2$.
\\

Let $$\rho_t:=\exp\Big\{\int_0^t(b_1)^*(\sigma^{-1})^*(\psi_r^2(x))\dd W_r-\frac{1}{2}\int_0^t(b_1)^*(\sigma\sigma^*)^{-1}b_1(\psi_r^2(x))\dd r\Big\},$$
where $\psi^2_t(x)$ is the flow generated by the solution to
\begin{align*}
    \dd \psi^2_t=b_2(\psi^2_t)\dd t+\sigma(\psi^2_t)\dd W_t,\quad \psi^2_0=x\in\mR^d.
\end{align*}
From \eqref{Khas-Kry-gener1} and \eqref{est:girsanov-novikov} we get for $T>1$ and any $\lambda>0$ 
\begin{align}\label{exp-Giranov}
    \mE\exp\Big(\lambda\int_0^T(b_1)^*(\sigma\sigma^*)^{-1}b_1(\psi_r^2(x))\dd r\Big)&\leq   2\exp\Big(T\big((K_1^2K_2)^{-1}(\Gamma\lambda)^2\Vert b_1\Vert_{\tilde L_p}^4+K_1^{-1}\lambda\Gamma\Vert b_1\Vert_{\tilde L_p}^2\big)\Big).
\end{align}
Therefore, $(\rho_t)_{t\geq0}$ is a martingale. Fix $T>0$ and let $\mP^\rho:=\rho_T\mP$. Girsanov's theorem
 and   H\"older's inequality show for each measurable set  $A \in C([0,T],\R^d)$
 \begin{align}\label{pst-1}
      \mP&(\psi|_{[0,T]}\in A)=\mP^\rho\big( \psi^2|_{[0,T]}\in A\big)
      \nonumber\\&=\mE\big[ \rho_T:\,\psi^2|_{[0,T]}\in A \big]
      \leq\big[\mE\rho_T^2\big]^{1/2}\mP\big(\psi^2|_{[0,T]}\in A\big)^{1/2} \nonumber\\&\leq\Big[ \mE\exp\Big(2\int_0^T(b_1)^*(\sigma^{-1})^*(\psi_r^2(x))\dd W_r-2\int_0^T(b_1)^*(\sigma\sigma^*)^{-1}b_1(\psi_r^2(x))\dd r\nonumber
   \\&\quad\quad \quad\quad +\int_0^T(b_1)^*(\sigma\sigma^*)^{-1}b^1(\psi_r^2(x))\dd r\Big)\Big]^{1/2}\Big[\mP\big( \psi^2|_{[0,T]}\in A  \big)\Big]^{1/2}
     \nonumber
     \\&\leq\Big(\mE\Big[\exp\Big(2\int_0^T(b_1)^*(\sigma^{-1})^*(\psi_r^2(x))\dd W_r-2\int_0^T(b_1)^*(\sigma\sigma^*)^{-1}b_1(\psi_r^2(x))\dd r\Big]^2\Big)^{1/4}\nonumber\\&\quad \Big[\mE\exp\Big(\int_0^t2(b_1)^*(\sigma\sigma^*)^{-1}b_1(\psi_r^2(x))\dd r\Big)\Big]^{1/4}\Big[\mP\big( \psi^2|_{[0,T]}\in A  \big)\Big]^{1/2}
    \nonumber 
    \\&\leq \Big[\mE\exp\Big(2\int_0^T(b_1)^*(\sigma\sigma^*)^{-1}b_1(\psi_r^2(x))\dd r\Big)\Big]^{1/4}\Big[\mP\big( \psi^2|_{[0,T]}\in A  \big)\Big]^{1/2}
      \nonumber
      \\&\leq 2 \exp\Big(T\big((K_1^2K_2)^{-1}\Gamma^2\Vert b_1\Vert_{\tilde L_p}^4+K_1^{-1}\Gamma\Vert b_1\Vert_{\tilde L_p}^2\big)\Big)\Big[\mP\big( \psi^2|_{[0,T]}\in A  \big)\Big]^{1/2}.
 \end{align}
If $A_i$ denotes the set inside $\mP$ on the left side of item i in the Lemma (i$=1,...,5)$, then
$$
 \mP(\psi|_{[0,T]}\in A_i)\leq 2\exp\Big(T\frac{\Gamma^2\Vert b_1\Vert_{\tilde L_p}^4+K_2^2\Gamma\Vert b_1\Vert_{\tilde L_p}^2}{{K_1K_2^2}}\Big)\Big[\mP\big( \psi^2|_{[0,T]}\in A_i  \big)\Big]^{1/2}
$$
finishing the first step in cases 1-5. It remains to estimate $\Big[\mP\big( \psi^2|_{[0,T]}\in A_i  \big)\Big]^{1/2}$. Inserting the estimate in  \cite[Proposition 4.2 a)]{DS} under $(U^\beta)$, we obtain statement
1. Inserting the estimate in  \cite[Proposition 4.5]{DS} under $(U^\beta)$, we obtain statement
2. Inserting the estimate in  \cite[Proposition 4.6]{DS} under $(U^\beta)$, we obtain statement
3. 
Inserting the estimate in  \cite[Proposition 4.2 b)]{DS} under $(U_\beta)$, we obtain statement
4. and inserting the estimate in  \cite[Proposition 4.3]{DS} under $(U_\beta)$, we obtain statement
5.

Finally we show items 6 and 7. Without loss of generality we assume $\frac{1}{\eta}<R$. For a ball $B_{R}$ with radius $R$ we can cover its boundary $\partial B_{R}$ by $N=N_{\epsilon}\leq C_d(\frac{R}{\epsilon})^{d-1}$ balls  centered on $\partial B_{R}$ for any $\epsilon\in(0,R]$. Here we take $\epsilon=\exp(-\kappa h(R))$ for some $\kappa>0$ which will be chosen later and we  label the balls by $L_1,\cdots,L_N$ with corresponding centers $x_1,\cdots,x_N$. Note that $$N\leq C_dR^{d-1}\exp\Big((d-1)\kappa h(R)\Big).$$
Then
\begin{align*}
   \mP_{R}
    \leq&N\max_{1\leq i\leq N}\Big[\mP\Big(|\psi_T(x_i)|\leq r_1+1,\inf_{t\in[0,T]}|\psi_t(x_i)|> r+1\Big)\\&+\mP\Big(\inf_{t\in[0,T]}|\psi_t(x_i)|\leq  r+1\Big)+\mP\Big(\sup_{t\in[0,T]}\text{diam }\psi_t(L_i)\geq1\Big)\Big]
   \\=:& N (P_1(R)+P_2(R)+P_3(R)).
\end{align*}
Case 4 gives us the following upper bound (note $T=h(R)$)
\begin{align*}
    P_1(R)&\leq  2\exp\Big(T\frac{\Gamma^2\Vert b_1\Vert_{\tilde L_p}^4+K_2^2\Gamma\Vert b_1\Vert_{\tilde L_p}^2}{{K_1K_2^2}}
    -\frac{1}{4}\Big(\frac{\sqrt{T}\beta_*(r+1)}{\sqrt{K_2}}-\frac{r_1+1-R}{\sqrt{K_2T}}\Big)_+^2\Big)
\\&=2\exp\Big(h(R)\frac{\Gamma^2\Vert b_1\Vert_{\tilde L_p}^4+K_2^2\Gamma\Vert b_1\Vert_{\tilde L_p}^2}{{K_1K_2^2}}
-\frac{h(R)}{4{K_2}}\Big(\beta_*(r+1)-\gamma-\frac{1}{h(R)}\Big)_+^2\Big).
\end{align*}
So 
\begin{align}\label{P1-log}
 \limsup_{R\rightarrow\infty}\frac{1}{h(R)}\log(N\mP_{1}(R))&\leq (d-1)\kappa+\frac{\Gamma^2\Vert b_1\Vert_{\tilde L_p}^4+K_2^2\Gamma\Vert b_1\Vert_{\tilde L_p}^2}{{K_1K_2^2}}-\frac{1}{4{K_2}}(\beta-\gamma)^2.
\end{align}
Case 5 shows for $r=(1-\eta)R$
\begin{align*}
    P_2(R)&\leq 2\exp\Big(T\frac{\Gamma^2\Vert b_1\Vert_{\tilde L_p}^4+K_2^2\Gamma\Vert b_1\Vert_{\tilde L_p}^2}{{K_1K_2^2}}-(R-r-1)\frac{\beta_*(r+1)}{K_2}\Big)
    \\&=2\exp\Big(h(R)\frac{\Gamma^2\Vert b_1\Vert_{\tilde L_p}^4+K_2^2\Gamma\Vert b_1\Vert_{\tilde L_p}^2}{{K_1K_2^2}}-(\eta R-1)\frac{\beta_*(r+1)}{K_2}\Big).
\end{align*}
Hence
\begin{align}\label{P2-log}
\frac{1}{h(R)}\log(N\mP_{2}(R))\leq(d-1)\kappa+\frac{\Gamma^2\Vert b_1\Vert_{\tilde L_p}^4+K_2^2\Gamma\Vert b_1\Vert_{\tilde L_p}^2}{{K_1K_2^2}}-\frac{\eta R-1}{h(R)}\frac{\beta_*(r+1)}{K_2}.
\end{align}
Furthermore, by \cref{two-point-prop},
%\cite[Proposition 4.4]{DS} shows that
\begin{align}\label{P3-log}
    \limsup_{R\rightarrow\infty}\frac{1}{h(R)}\log P_3(R)\leq -\kappa^{4/3}c_1^{-1/3}\text{ with } \kappa>4c_1d^3
\end{align}
where  $c_1$ is taken from \eqref{c0c1c2} with $b$ replaced by $b_2$. Therefore, by \eqref{P1-log}, \eqref{P2-log} and \eqref{P3-log},  it follows that, for $\kappa>4c_1d^3$,
\begin{align}\label{PR}
\limsup_{R\rightarrow\infty}&\frac{1}{h(R)}\log\mP_{R}\nonumber\\\leq& \limsup_{R\rightarrow\infty}\frac{1}{h(R)}\log(N\mP_{1}(R)+ N\mP_{2}(R)+N\mP_{3}(R))
\nonumber\\\leq&2(d-1)\kappa+2\frac{\Gamma^2\Vert b_1\Vert_{\tilde L_p}^4+K_2^2\Gamma\Vert b_1\Vert_{\tilde L_p}^2}{{K_1K_2^2}}-\frac{1}{4{K_2}}(\beta-\gamma)^2+(d-1)\kappa -\kappa^{4/3}c_1^{-1/3}.
\end{align}
Notice that  $\beta-\gamma>\beta_0+\eta\geq 4\frac{\Vert b_1\Vert_{\tilde L_p}^2\Gamma+K_2\Vert b_1\Vert_{\tilde L_p}\sqrt\Gamma}{{\sqrt{K_1K_2}}}.$ If we choose $\kappa\geq 3c_1(d-1)^3$ initially, then get  
$$\limsup_{R\rightarrow\infty}\frac{1}{h(R)}\log\mP_{R}<0.$$
Therefore case 6 holds. \\
 
 We show case 7 in a similar way.
 %Easily it can be checked that specific $h(R)=R^\iota$ satisfies the conditions on $h$ in case 7. 
 We again cover $\partial B_{r_1}$ by $N\leq C_d{r_1}^{d-1}e^{\kappa(d-1)T}$ balls  centered on $\partial B_{r_1}$ for any with radius $e^{-\kappa T}$ for some $\kappa>0$ chosen later. Label the balls by $L_1,\cdots,L_N$ with corresponding centers $x_1,\cdots,x_N$. Then 
 \begin{align*}
     \mP_R\leq& N\max_i\Big[\mP\Big(|\psi_T(x_i)|\geq R+1,\inf_{t\in[0,T]}|\psi_t(x_i)|> r+1\Big)\\&+\mP\Big(|\psi_T(x_i)|\geq R+1,\inf_{t\in[0,T]}|\psi_t(x_i)|\leq  r+1\Big)+\mP\Big(\sup_{t\in[0,T]}\text{diam }\psi_t(L_i)\geq1\Big)\Big]
   \\=:& N (P_1(R)+P_2(R)+P_3(R)).
 \end{align*}
 From case 1 we then get (note $T=h(R)$)
 \begin{align*}
     P_1(R)\leq& 2\exp\Big(h(R)\frac{\Gamma^2\Vert b_1\Vert_{\tilde L_p}^4+K_2^2\Gamma\Vert b_1\Vert_{\tilde L_p}^2}{{K_1K_2^2}}
-\frac{h(R)}{4{K_2}}\Big(\frac{R+1-r_1}{h(R)}-\beta^*(r+1)\Big)_+^2\Big)
\\\leq & 2\exp\Big(h(R)\frac{\Gamma^2\Vert b_1\Vert_{\tilde L_p}^4+K_2^2\Gamma\Vert b_1\Vert_{\tilde L_p}^2}{{K_1K_2^2}}
-\frac{h(R)}{4{K_2}}(\frac{1}{h(R)}-\gamma-\beta)^2\Big).
 \end{align*}
 Therefore,
 \begin{align}\label{P1-upper}
  \frac{1}{h(R)}\log  P_1(R)\leq \frac{\Gamma^2\Vert b_1\Vert_{\tilde L_p}^4+K_2^2\Gamma\Vert b_1\Vert_{\tilde L_p}^2}{{K_1K_2^2}}
-\frac{1}{4{K_2}}(\frac{1}{h(R)}-\gamma-\beta)^2.
 \end{align}
Analogously, case 2 implies for $R$ such that $r=(1-\eta)R>r_0$ where $\beta^*(r_0)<0$, 
\begin{align}\label{P2-upper}
  \frac{1}{h(R)}\log  P_2(R)\leq \frac{\Gamma^2\Vert b_1\Vert_{\tilde L_p}^4+K_2^2\Gamma\Vert b_1\Vert_{\tilde L_p}^2}{{K_1K_2^2}}
-\frac{(R-r)^2}{16{K_2}h(R)}.
 \end{align}
By \eqref{P1-upper}, \eqref{P2-upper} and \eqref{P3-log} we obtain
 \begin{align}\label{case7-P}
 &\frac{1}{h(R)}\log\mP_R\nonumber\\&\leq 3(d-1)\kappa+2\frac{\Gamma^2\Vert b_1\Vert_{\tilde L_p}^4+K_2^2\Gamma\Vert b_1\Vert_{\tilde L_p}^2}{K_1K_2^2}-\frac{1}{4{K_2}}(\frac{1}{h(R)}-\beta-\gamma)^2-\frac{(R-r)^2}{16{K_2}h(R)} -\kappa^{4/3}c_1^{-1/3}
 \end{align}
 and 
 \begin{align*}
     \limsup_{R\rightarrow\infty}\frac{1}{h(R)}\log\mP_R\leq 3(d-1)\kappa+2\frac{\Gamma^2\Vert b_1\Vert_{\tilde L_p}^4+K_2^2\Gamma\Vert b_1\Vert_{\tilde L_p}^2}{K_1K_2^2}-\frac{1}{4{K_2}}(-\beta-\gamma)^2-\kappa^{4/3}c_1^{-1/3}.
 \end{align*}
Under $(U^\beta)$,
\begin{align*}
     (-\gamma-\beta)^2 \geq (-\beta_0-\eta)^2\geq 16 K_2\frac{\Gamma^2\Vert b_1\Vert_{\tilde L_p}^4+K_2^2\Gamma\Vert b_1\Vert_{\tilde L_p}^2}{{K_1K_2^2}}.
 \end{align*}
Hence, choosing $\kappa\geq 3c_1(d-1)^3$ above, we conclude that   
$\limsup_{R\rightarrow\infty}\frac{1}{h(R)}\log\mP_{R}<0.$
 \end{proof}
 Now we are ready to state the first main theorem of this section.
 \begin{theorem}\label{non-attractor}
   Let \cref{Ass} hold. Further assume that there exist vector fields $b_1$ and $ b_2$ such that $b=b_1+b_2$ with $b_1\in \tilde L_p(\mR^d)$. Let $(\psi_t(x))_{t\geq0}$ denote the flow generated by the solution to  \eqref{sde0}.  Let  $\Gamma:=C_{\mathrm{Kry}}(\frac{p}{2})\Big(\big(\frac{K_2}{K_1}\big)^{\frac{4d^2}{1-d/\rho}}+\big(\frac{\Vert\nabla\sigma\Vert_{\tilde L_\oops}^2}{K_1}\big)^{\frac{4d^2}{1-d/\rho}}+\big(\frac{\Vert b_2\Vert_{\tilde L_{p}}}{K_1}\big)^{\frac{4d}{1-d/p}}\Big)$  where  $C_{\mathrm{Kry}}(\frac{p}{2})$  is  from \eqref{est-Krlov} with $q=\frac{p}{2}$ depending on $p,\rho$ and $d$ only.  If $b_2$ satisfies \cref{Ass-onepoint} $(U_\beta)$ for 
   $$\beta>\beta_0:= 4\frac{\Vert b_1\Vert_{\tilde L_p}^2\Gamma+K_2\Vert b_1\Vert_{\tilde L_p}\sqrt\Gamma}{{\sqrt{K_1K_2}}},$$
   then for any $\gamma\in[0,\beta-\beta_0)$ we have
   \begin{align}\label{pro-lowerbound}
   \lim_{r\rightarrow\infty}\mP\Big(B_{\gamma t}\subset\psi_t(B_r)\quad\forall \quad t\geq0\Big)=1.
   \end{align}
 \end{theorem}
 \begin{proof}
   For $\gamma\in[0,\beta-\beta_0)$, let $\eta\in(0,\frac{1}{2})$ such that $\gamma+\eta<\beta-\beta_0$. Let $R_0\geq 2$, $R_{i+1}=R_i+\gamma h(R_i)$ by iteration, where $h:[1,\infty)\rightarrow[1,\infty)$ is strictly increasing and $\lim_{x\rightarrow\infty}\frac{h(x)}{x}=0$ and $\lim_{x\rightarrow\infty}\frac{\log x}{h(x)}=0$. For $i=0,1,\cdots,$ take $r_i=(1-\eta)R_i$, $\bar r_i=R+\gamma h(R_i)$. Define
   $$\mP_{R_i}:=\mP\Big[\Big(B_{\bar r_i}\nsubseteq\psi_T(B_{R_i})\Big)\cup\cup_{t\in[0,T]}\Big(B_{r_i}\nsubseteq\psi_t(B_{R_i})\Big)\Big].$$
   Then \cref{lemm-mul} case 6 shows that 
   \begin{align*}
     \sum_{i=0}^{\infty}\mP_{R_i}<\infty, \quad\text{if}\quad  \sum_{i=0}^{\infty}\exp(-\kappa h(R_i))<\infty, \quad \kappa>0.
   \end{align*}
   If we take $h(R_i)=R_i^\alpha$ for some $\alpha\in(0,1)$, then  Borel-Cantelli Lemma and time-homogeneity of flow $\psi$ yield the result \eqref{pro-lowerbound}.
 \end{proof}
  Finally, we state the following  theorem on the existence of random attractors.
  \begin{theorem}\label{exist-Random-attractor}
    Let \cref{Ass} hold. Further assume that there exist vector fields $b_1$ and $ b_2$ such that $b=b_1+b_2$ with $b_1\in \tilde L_p(\mR^d)$. Let $(\psi_t(x))_{t\geq0}$ denote the flow generated by the solution to  \eqref{sde0}.  Let  $\Gamma:=C_{\mathrm{Kry}}(\frac{p}{2})\Big(\big(\frac{K_2}{K_1}\big)^{\frac{4d^2}{1-d/\rho}}+\big(\frac{\Vert\nabla\sigma\Vert_{\tilde L_\oops}^2}{K_1}\big)^{\frac{4d^2}{1-d/\rho}}+\big(\frac{\Vert b_2\Vert_{\tilde L_{p}}}{K_1}\big)^{\frac{4d}{1-d/p}}\Big)$  where  $C_{\mathrm{Kry}}(\frac{p}{2})$  is  from \eqref{est-Krlov} with $q=\frac{p}{2}$ depending on $p, \rho$ and $d$ only.  If $b_2$ satisfies \cref{Ass-onepoint} $(U^\beta)$ for $$\beta<-\beta_0:= -4\frac{\Vert b_1\Vert_{\tilde L_p}^2\Gamma+K_2\Vert b_1\Vert_{\tilde L_p}\sqrt\Gamma}{{\sqrt{K_1K_2}}},$$
   then, for any $\gamma\in[0,-\beta-\beta_0)$, we have
   \begin{align}\label{pro-upperbound}
   \lim_{r\rightarrow\infty}\mP\Big(B_{\gamma t}\subset\psi_{-t,0}^{-1}(B_r)\quad\forall \quad t\geq0\Big)=1.
   \end{align}
   In particular, $\psi$ has a random attractor.
  \end{theorem}
  \begin{proof}
  The existence of an attractor is an easy observation from \cref{flow-attractor} if we have \eqref{pro-upperbound}. So we only need to show \eqref{pro-upperbound}. The argument  is essentially  the same as  \cite[Proof of Theorem 3.1 a)]{DS}. We give the outline of the proof emphasising those arguments  which are different.  \\
 
For $\gamma\in[0,-\beta-\beta_0)$, let $\eta\in(0,\frac{1}{2})$ such that $\gamma+\eta<-\beta-\beta_0$. Let $h(y)=y^\alpha$   for some $\alpha\in(0,\frac{1}{3})$. Notice that such $h$ is strictly increasing and $\lim_{y\rightarrow\infty}\frac{h(y)}{y}=0$ and $\lim_{y\rightarrow\infty}\frac{\log y}{h(y)}=0$. For $T\in(1,\infty)$, take $R:= T^{1/\alpha}$, $r_1=R+\gamma T$ and $r=(1-\eta)R$. Let  $(\phi_{s,T}(x))_{s\leq T}$ denote the flow starting from $x$ at initial time $s$. We define     
$$\mP_{R}:=\mP\Big[\Big(B_{r_1}\nsubseteq\psi_T^{-1}(B_{R})\Big)\cup\cup_{t\in[0,T]}\Big(B_{r}\nsubseteq\phi_{t,T}^{-1}(B_{R})\Big)\Big].$$  
Once we show that 
\begin{align}
    \label{aim}
    \lim_{R\rightarrow\infty}\frac{1}{h(R)}\log \mP_R<0,
\end{align}
then, by the same argument as in the proof of \cref{non-attractor}, we can finish the proof by the Borel-Cantelli Lemma and time-homogeneity of the flow $\psi$. 

To show \eqref{aim}, notice that 
\begin{align*}
   \mP_{R}&\leq \mP\Big[\cup_{|x|=r_1}\Big((|\psi_T(x)|\geq R)\cap(\inf_{t\in[0,T]}|\psi_{t}(x)|\geq r)\Big)\Big]+\mP\Big(\sup_{|x|=r}\sup_{t\in[0,T]}|\psi_{t,T}(x)|\geq R\Big) 
   \\&=:P_1(R)+P_2(R).
\end{align*}
For $P_1(R)$, we get from \cref{lemm-mul}, case 7 that  (note $T=R^\alpha=h(R)$)
\begin{align*}
    \lim_{R\rightarrow\infty}\frac{1}{h(R)}\log P_1(R)<0.
\end{align*}
In the following we show 
\begin{align}
    \label{P2}
     \lim_{R\rightarrow\infty}\frac{1}{h(R)}\log P_2(R)=-\infty,
\end{align}
which is sufficient to get \eqref{aim}. 

Let $\xi_s:=(\sup_{|x|=r}|\psi_{s,T}(x)|-r)_+$, $\zeta_s:=(\sup_{|x|=r+R\eta/2}|\psi_{s,T}(x)|-r)_+$.  Then, as shown in \cite[p.1205-1206]{DS}, we have
\begin{align*}
    \limsup_{R\rightarrow\infty}&\frac{1}{h(R)}\log P_2(R)\\\leq & \limsup_{R\rightarrow\infty}\frac{1}{h(R)}\log\max_{s\in[1,T]}\Big[\mP\Big(\zeta_s\geq \eta R\Big)+\mP\Big(\sup_{t\in[s-1,s]}\sup_{|x|=r}|\psi_{t,s}(x)|\geq r+\frac{\eta}{2}R\Big)\Big]
    \\&:=\limsup_{R\rightarrow\infty}\frac{1}{h(R)}\log\max_{s\in[1,T]}(P_{2,1}(s,R)+P_{2,2}(s,R)).
\end{align*}
To estimate $P_{2,1}(s,R)$, for fixed $0\leq s\leq T$, denote $r_0:=r+\frac{\eta}{2}R$, we cover $\partial B_{r_0}$ by $N\leq C_dr_0^{d-1}e^{\kappa(d-1)T}$ balls of radius $e^{-\kappa T}$ centered on $\partial B_{r_0}$ with $\kappa<\frac{c_1d^2}{3(d-1)}$ (the same choice as in the proof of \cref{lemm-mul} case 7. Label the balls by $L_1,\cdots, L_N$ and their centers correspondingly by $x_1,\cdots,x_N$. Then for a number $r_2$ such that $\beta^*(r_2)<0$ where $\beta^*(r_2)$ is from \eqref{beta-up}, we have 
\begin{align*}
  P_{2,1}(s,R) 
  \leq N\max_i\Big[&\mP\Big(|\psi_{s,T}(x_i)|\geq r+\eta R-1\Big)+\mP\Big(\text{diam } \psi_{s,T}(L_i)\geq 1\Big)\Big]  
 \\ \leq N\max_i\Big[&\mP\Big(|\psi_{s,T}(x_i)|\geq r+\eta R-1,\inf_{s\leq t\leq T}|\psi_{s,t}(x_i)|> r_2\Big)\\&+\mP\Big(|\psi_{s,T}(x_i)|\geq r+\eta R-1,\inf_{s\leq t\leq T}|\psi_{s,t}(x_i)|\leq r_2\Big)+\mP\Big(\text{diam } \psi_{s,T}(L_i)\geq 1\Big)\Big].  
\end{align*}
By the same argument from \cref{lemm-mul} case 7 \eqref{case7-P} with $h(R)=R^\alpha=T$, and \cref{lemm-mul} case 2, and \cref{two-point-prop} we get 
\begin{align*}
 \limsup_{R\rightarrow\infty} \frac{1}{h(R)}\log \max_{s\in[1,T]}P_{2,1}(s,R)  =-\infty.
\end{align*}
Up to here, in order to get \eqref{P2}, we only  need to show 
\begin{align}
    \label{P22} \limsup_{T\rightarrow\infty}\frac{1}{T}\log \max_{s\in[1,T]} P_{2,2}(s,T^{1/\alpha})=-\infty.
\end{align}
In \cite[Proof of Theorem 3.1 a)]{DS}, this is shown by using three statements: \cite[(4.7)]{DS}, \cite[Proposition 4.5]{DS} and \cite[Proposition 4.6]{DS}. In our setting, we already showed the second and the third statements: these are \cref{lemm-mul} case 2 and case 3 correspondingly. 
% Starting from  here there is only one difference between our  proof and  \cite[Proof of Theorem 3.1 a)]{DS}, namely  \cite[(4.7)]{DS}, the left argument would be the same with replacing  \cite[Proposition 4.5]{DS} and \cite[Proposition 4.6]{DS} by   \cref{lemm-mul} case 2 and case 3 correspondingly.  
Therefore it is sufficient to show the estimate  corresponding to  \cite[(4.7)]{DS} in our setting.
In order to do so we first apply Girsanov Theorem as we did in \cref{lemm-mul}. Let $$\rho_t:=\exp\Big(\int_0^tb^*(\sigma^{-1})^*(\phi_r(x))\dd W_r-\frac{1}{2}\int_0^tb^*(\sigma\sigma^*)^{-1}b(\phi_r(x))\dd r\Big),$$
where $\phi_t(x)$ is the flow generated by the solution to
\begin{align*}
   \dd \phi_t=\sigma(\phi_t)\dd W_t,\quad \phi_0(x)=x\in\mR^d.
\end{align*}
Following from \eqref{Khas-Kry-gener1} we get for $T>1$ and any $\lambda>0$ 
\begin{align*}
    \mE\exp\Big(\lambda\int_0^Tb^*(\sigma\sigma^*)^{-1}b(\phi_r(x))dr\Big)&\leq  \exp\Big(T\frac{\Vert b\Vert_{\tilde L_p}^4(\lambda\Gamma')^2+K_2^2\Vert b\Vert_{\tilde L_p}^2\lambda\Gamma'}{{K_1K_2^2}}\Big)
\end{align*}
where $\Gamma'=C_{\mathrm{Kry}}(\frac{p}{2})\Big(\big(\frac{K_2}{K_1}\big)^{\frac{4d^2}{1-d/\rho}}+\big(\frac{\Vert\nabla\sigma\Vert_{\tilde L_\oops}^2}{K_1}\big)^{\frac{4d^2}{1-d/\rho}}\Big)$ and $C_{\mathrm{Kry}}(\frac{p}{2})$ is from \eqref{est-Krlov} with $p=\frac{p}{2}$ and $b=0$.
Therefore $(\rho_t)_{t\geq0}$ is a martingale. Let $\mP^\rho:=\rho_1\mP$. As we already did in \eqref{pst-1}, by Girsanov theorem and H\"older's inequality,  for $\epsilon>0$, for any $x,z\in\mR^d$, 
\begin{align}\label{PRO-RHO}
  \mP\Big(|\psi_{t+\frac{1}{2^n},1}(x)&-\psi_{t+\frac{1}{2^n},1}(z)|\geq \frac{\epsilon}{2}\Big)
  \nonumber\\=&\mP^\rho\Big(|\phi_{t+\frac{1}{2^n},1}(x)-\phi_{t+\frac{1}{2^n},1}(z)|\geq \frac{\epsilon}{2}\Big)
  \nonumber\\=&\mE[\rho_1\mathbb{I}_{\small\left\{|\phi_{t+\frac{1}{2^n},1}(x)-\phi_{t+\frac{1}{2^n},1}(z)|\geq \frac{\epsilon}{2}\small\right\}}]
\nonumber\\\leq& 2\exp\Big(\frac{\Vert b\Vert_{\tilde L_p}^4\Gamma'^2+K_2^2\Vert b\Vert_{\tilde L_p}^2\Gamma'}{{K_1K_2^2}}\Big)\Big[\mP\Big(|\phi_{t+\frac{1}{2^n},1}(x)-\phi_{t+\frac{1}{2^n},1}(z)|\geq \frac{\epsilon}{2}\Big)\Big]^{1/2}.
\end{align}
Let $B_t(x):=W_{\int_{t}^1|\sigma|^2(\phi_r(x))dr}$, then by time change and the fact that for $\kappa_1,\kappa_2\in \mR$
$$\mP\Big(W_t\geq \kappa_1\Big)\leq \frac{1}{2}e^{-\frac{\kappa_1^2}{2t}},\quad \mP(\sup_{s\leq t}W_s\geq \kappa_2)\leq e^{-\frac{\kappa_2^2}{2t}},$$ we know  for $x,z\in\mR^d$ and $|x-z|\leq \delta$ with $\delta>0$
\begin{align*}
    \Big[&\mP\Big(|\phi_{t+\frac{1}{2^n},1}(x)-\phi_{t+\frac{1}{2^n},1}(z)|\geq \frac{\epsilon}{2}\Big)\Big]^{1/2}\nonumber\\\leq&\Big[\mP\Big(|B_{t+\frac{1}{2^n}}(x)-B_{t+\frac{1}{2^n}}(z)|\geq \frac{\epsilon}{2}-\delta\Big)\Big]^{1/2}
\nonumber\\\leq &\Big[ \exp\Big(-\frac{(\epsilon-2\delta)^2}{4}\frac{1}{2(\int_{t+\frac{1}{2^n}}^1|\sigma|^2(\phi_r(x))\dd r-\int_{t+\frac{1}{2^n}}^1|\sigma|^2(\phi_r(x))\dd r)}\Big)\Big]^{1/2}
\nonumber\\\leq & \exp\Big(-\frac{(\epsilon-2\delta)^2}{16}\frac{1}{K_2-K_1}\Big).
\end{align*}
Accordingly by \eqref{PRO-RHO}  for any $\epsilon,\delta>0$ and for any $x,z\in\mR^d$ with $|x-z|\leq \delta$ we have
\begin{align*}
    \mP\Big(|\psi_{t+\frac{1}{2^n},1}(x)-\psi_{t+\frac{1}{2^n},1}(z)|\geq \frac{\epsilon}{2}\Big)& \leq 2\exp\Big(\frac{\Vert b\Vert_{\tilde L_p}^4\Gamma'^2+K_2^2\Vert b\Vert_{\tilde L_p}^2\Gamma'}{{K_1K_2^2}}-\frac{(\epsilon-2\delta)^2}{16}\frac{1}{K_2-K_1}\Big)\\&\lesssim \exp\Big(-\frac{(\epsilon-2\delta)^2}{16}\frac{1}{K_2-K_1}\Big)
\end{align*}
corresponding to \cite[(4.7)]{DS}. Applying the argument from \cite[Proof of Theorem 3.1 a)]{DS} we get that $P_{2,2}(s,T^{1/\alpha})$ decays super exponentially in $T$, therefore \eqref{P22} holds. The proof  is complete.
  \end{proof}
 \appendix
 \section{Bounds for solutions of elliptic PDEs}\label{A}
 Consider the following elliptic equation on $\mR^d$ (recall the summation convention):
 \begin{align}
\label{PDE}
\lambda u- a_{ij}\partial_{ij}u+b\cdot \nabla u=f,
 \end{align}
 where $\lambda>0$, $a(\cdot):\mR^d\rightarrow\mR^d\otimes\mR^d$ is a symmetric matrix-valued Borel measurable function, and 
 $b(\cdot):\mR^d\rightarrow\mR^d$ and  $f:\mR^d\rightarrow\mR$ are  Borel measurable functions such that $f\in\tilde L_p(\mR^d)$ with $p\in(1,\infty)$. The definition of the solution to  equation \eqref{PDE} is as follows:
 \begin{definition}
 Let $\lambda>0$. We call $u\in\tilde{H}^{2,p}$ a strong solution to \eqref{PDE} if for a.e. $x\in\mR^d$,
 \begin{align*}
     \lambda u(x)-a_{ij}(x)\partial_{ij} u(x)+b(x)\cdot \nabla u(x)=f(x).
 \end{align*}
 \end{definition}
 We assume
 \begin{assumption}\label{Assumption-pde}
  \begin{itemize}
  \item[($H^a$)] there exist  $0<K_1\leq K_2$ such that for all $x\in\mathbb{R}^d$,
  \begin{align}\label{uni-elliptic}
    K_1|\zeta|^2\leq\langle a(x)\zeta,\zeta\rangle\leq K_2|\zeta|^2,\quad \forall \zeta\in\mathbb{R}^d,
  \end{align}
  and $a(\cdot)$ is $\alpha$-H\"older continuous with
\begin{align}
     \label{a-delta}
   \omega_\alpha(a):=\sup_{x,y\in\mathbb{R}^d,x\neq y,|x-y|\leq 1}\frac{\Vert a(x)-a(y)\Vert}{|x-y|^\alpha}<\infty
 \end{align}
 for some $\alpha\in(0,1]$.
      \item[($H^b$)] $b\in \tilde L_{p_1}(\mR^d)$ for some $p_1\in(d,\infty]$.
  \end{itemize}
 \end{assumption}
 
 In this section we will show estimates of the solution of the elliptic PDE above.  Such estimates were obtained in  \cite[Theorem 3.3]{Zhangzhao2018} in the case where $a$ is uniformly elliptic and uniformly continuous and $b\in L_{p_1}$ for some $p_1>d$. These estimates were, however, not  
 %they obtained the corresponding \emph{a priori} estimates assuming $a$ is uniformly elliptic and uniformly continuous and $b\in L_{p_1}$ for $p_1>d$.   
 explicit in terms of the  coefficients $a$, $b$ and $f$. We prove the following theorem which shows this dependence since we need it in the main text but it may also be of independent interest.
 
 \begin{theorem}
\label{aprior-est-local}
Suppose \cref{Assumption-pde} holds.  There exists a constant $C_0>0$  depending on $p$, $p_1$, $\alpha$ and $d$ only, such that for $\lambda\geq C_0K_1\Big(\frac{K_2^2}{K_1^2}(\frac{K_1+\omega_\alpha(a)}{K_1})^{\frac{2}{\alpha}}+(\frac{K_1+\omega_\alpha(a)}{K_1})^{\frac{d}{\alpha}\frac{2}{1-d/p_1}}(\frac{\Vert b\Vert_{\tilde L_{p_1}}}{K_1})^{\frac{2}{1-d/p_1}}\Big)$ and   for any $ f\in\tilde L_p(\mR^d)$ with $p\in(d/2\vee1,p_1]$, there is a unique solution $u\in \tilde H^{2,p}$ to \eqref{PDE}.  Further, for $p'\in[1,\infty]$ there exists 
a constant $C$ depending on $\alpha,p,d,p'$ and $p_1$ only, such that 
\begin{align}
    \label{aprior-estimate-pde}
    &\Vert \nabla^2u\Vert_{\tilde L_{p}}\leq C\frac{1}{K_1}\Big(1+\frac{\omega_\alpha(a)}{K_1}\Big)^{d/\alpha}\Vert f\Vert_{\tilde L_p},\nonumber\\&\lambda^{(1+\frac{d}{p'}-\frac{d}{p})/2}\Vert \nabla u\Vert_{\tilde L_{p'}}\leq C{K_1}^{(\frac{d}{p'}-\frac{d}{p}-1)/2}\Big(1+\frac{\omega_\alpha(a)}{K_1}\Big)^{d/\alpha}\Vert f\Vert_{\tilde L_p}\quad\text{ if }\quad 1+\frac{d}{p'}-\frac{d}{p}>0,\nonumber\\& \lambda^{(2+\frac{d}{p'}-\frac{d}{p})/2}\Vert u\Vert_{\tilde L_{p'}}\leq C{K_1}^{(\frac{d}{p'}-\frac{d}{p})/2}
   \Big(1+\frac{\omega_\alpha(a)}{K_1}\Big)^{d/\alpha}\Vert f\Vert_{\tilde L_p} \quad\text{ if }\quad 2+\frac{d}{p'}-\frac{d}{p}>0.
\end{align}
\iffalse, and
\begin{align}
    \label{gamma0}
    \gamma_0:=\frac{C_{0}\frac{1}{\sqrt{K_1}}(1+\frac{C_0K_2}{\lambda_0 -C_0K_2})\Vert  b\Vert_{\tilde L_{p_1}}}{\lambda^{(1-\frac{d}{p_1})/2}-C_0\frac{1}{\sqrt{K_1}}(1+\frac{C_0K_2}{\lambda_0 -C_0K_2})\Vert  b\Vert_{\tilde L_{p_1}})}.
\end{align}\fi
\iffalse In particular, we have
\begin{align}
    \label{gradient-lapace}
    \Vert \nabla^2u\Vert_{\tilde L_p}\leq C\frac{1}{K_1}(1+\gamma_0)\Vert f\Vert_{\tilde L_p},\quad \lambda^{\Big(1+\frac{d}{p'}-\frac{d}{p}\Big)/2} \Vert \nabla u\Vert_{\tilde L_{p'}}\leq C\frac{1}{\sqrt{K_1}}(1+\gamma_0)\Vert f\Vert_{\tilde L_p}.
\end{align}\fi
 \end{theorem}
 \iffalse
 We first  introduce the following result which is considered in $L_p$ space.
 \begin{lemma}
\label{aprior-est-global}
Assume $\lambda>1$, let $b\in \tlpone$ for some $p_1>d$. For any $p\in(d/2\vee1,p_1]$, there is a large constant $\rho$ independent of    $\lambda$ and $b$ with $\rho \lambda^{(\frac{d}{p_1}-1)/2}\Vert b\Vert_{\tlpone}\leq \frac{1}{2}$ such that for any $f\in \tlp$, there is a unique solution $u\in \tilde H^{2,p}$ to \eqref{PDE} so that
\begin{align}
    \label{aprior-estimate-pde}
    \Vert u\Vert_{\tilde H^{2,p}}\leq C\Vert f\Vert_{ \tlp},\quad \lambda^{\Big(2-\alpha+\frac{d}{p'}-\frac{d}{p}\Big)/2}\Vert u\Vert_{\tilde H^{\alpha,p'}}\leq C\Vert f\Vert_{ \tlp},
\end{align}
where $\alpha\in[0,2)$ and $p'\in[1,\infty]$ with $\frac{d}{p}<2-\alpha+\frac{d}{p'}$. Here the constant $C$ is independent of $\lambda$, $f$ and $b$.
 \end{lemma}
 \begin{proof}

 \end{proof}\fi

 \begin{proof}
Assume $u\in \tilde{H}^{2,p}$ is a solution to \eqref{PDE}. We first show the \emph{a  priori} estimates \eqref{aprior-estimate-pde}. Then the {\em continuity method}, as shown in \cite{Krylov},  is a standard way to conclude the existence and uniqueness of the solution to \eqref{PDE} for those $\lambda$ for which \eqref{aprior-estimate-pde} holds.    We divide the proof into three steps. Note  that all  positive constants $C_i,i=1,\cdots$  appearing in the proof only depend on  $d, p, p_1, p',\alpha$ (and not on $\lambda$, $f$, $b$, $a$, and $\omega_\alpha(a)$).\\

 \emph{Step 1. Assume that $a$ is a constant (positive definite)  matrix, $b=0$ and $f\in L_p$.} \\
 
For $\lambda>0$, let $v\in H^{2,p}$ be the solution to the following equation
 $$  \lambda v-\Delta v=\tilde f,\quad \tilde f(x):=f(\sigma x),\quad x\in\mR^d,$$
  where $\sigma$ is the unique positive definite matrix satisfying $\sigma\sigma^*=a$.   Then $v= (\lambda-\Delta)^{-1}\tilde f$ is the unique solution in $H^{2,p}$. From \cite[(3.3)]{Zhangzhao2018} we know that, for each $p'\in[1,\infty]$, there are constants $C_1, C_2, C_3$  such that 
 \begin{align}\label{est:resolvent-eq}
     &\Vert \nabla^2v\Vert_{L_{p}}\leq C_1\Vert \tilde f\Vert_{L_p},\quad \nonumber\\&\lambda^{(1+\frac{d}{p'}-\frac{d}{p})/2} \Vert \nabla v\Vert_{L_{p'}}\leq C_2\Vert \tilde f\Vert_{L_p},\quad\text{ if }\quad 1+\frac{d}{p'}-\frac{d}{p}>0,  \nonumber\\&\lambda^{(2+\frac{d}{p'}-\frac{d}{p})/2} \Vert v\Vert_{L_{p'}}\leq C_3\Vert \tilde f\Vert_{L_p}\quad\text{ if }\quad 2+\frac{d}{p'}-\frac{d}{p}>0.
 \end{align}
    Let $u(x):=v(\sigma^{-1} x)$, i.e. $v(x)=u(\sigma x)$. Observe that
 \begin{align*}
    \partial_iv(x)=\partial_ku(\sigma x)\sigma_{ki},\quad \partial_{ij}v(x)=\partial_{kr}u(\sigma x)\sigma_{ki}\sigma_{rj}.\quad
 \end{align*}
 Therefore
 \begin{align*}
    (\lambda -\Delta) v(x)=(\lambda -a_{ij}\partial_{ij})u(\sigma x)
 \end{align*}
 and hence $u$ solves \eqref{PDE}. Uniqueness of a solution under the conditions of Step 1 holds since 
 the map $v \mapsto u$ is  a bijection between solutions of the corresponding PDEs.
 Considering
 $$  \frac{1}{K_1^{p}}\Vert \nabla^2 v\Vert_{L_{p}}^{p}
 \geq\text{det}\sigma^{-1}  \Vert \nabla^2 u\Vert_{L_{p}}^{p},\quad  \frac{1}{K_1^{p'/2}}\Vert \nabla v\Vert_{L_{p'}}^{p'}\geq\text{det}\sigma^{-1}  \Vert \nabla u\Vert_{L_{p'}}^{p'},\quad\Vert \tilde f\Vert_{L_p}^{p}=\text{det}\sigma^{-1}\Vert  f\Vert_{L_p}^{p},$$
then \eqref{est:resolvent-eq} yields 
 \begin{align}
 \Vert \nabla^2u\Vert_{L_{p}}&\leq C_1\frac{1}{K_1}\Vert f\Vert_{L_p},\quad \nonumber \\
      \lambda^{(1+\frac{d}{p'}-\frac{d}{p})/2} \Vert \nabla u\Vert_{L_{p'}}&\leq C_2(\text{det}\sigma^{-1} )^{\frac{1}{p}-\frac{1}{p'}}\frac{1}{\sqrt{K_1}}\Vert  f\Vert_{L_p}, \quad\text{ if }\quad 1+\frac{d}{p'}-\frac{d}{p}>0,
  \nonumber \\ \lambda^{(2+\frac{d}{p'}-\frac{d}{p})/2} \Vert u\Vert_{L_{p'}}&\leq C_3(\text{det}\sigma^{-1} )^{\frac{1}{p}-\frac{1}{p'}}\Vert  f\Vert_{L_p}\quad\text{ if }\quad 2+\frac{d}{p'}-\frac{d}{p}>0.\label{est:first-second}
 \end{align}
We know that $\text{det}\sigma=\prod_{i=1}^d\sqrt{\lambda_i}$ where $\lambda_i>0,i=1,\cdots,d,$ are the eigenvalues of $a$. From \eqref{uni-elliptic} we get
$\lambda_i\in[K_1,K_2]$. Therefore
\begin{align}\label{det-sigma}
  \text{det}\sigma^{-1}\in[K_2^{-\frac{d}{2}},K_1^{-\frac{d}{2}}].  
\end{align}
Using    \eqref{est:first-second} and \eqref{det-sigma}, we finally get
  \begin{align}\label{est-resl}
    & \Vert \nabla^2u\Vert_{L_{p}}\leq C_1\frac{1}{K_1}\Vert f\Vert_{L_p},\quad \nonumber \\ &\lambda^{(1+\frac{d}{p'}-\frac{d}{p})/2} \Vert \nabla u\Vert_{L_{p'}}\leq C_2{K_1}^{(\frac{d}{p'}-\frac{d}{p}-1)/2}\Vert  f\Vert_{L_p} \quad\text{ if }\quad 1+\frac{d}{p'}-\frac{d}{p}>0,
  \nonumber \\ &\lambda^{(2+\frac{d}{p'}-\frac{d}{p})/2} \Vert u\Vert_{L_{p'}}\leq C_3{K_1}^{(\frac{d}{p'}-\frac{d}{p})/2}\Vert  f\Vert_{L_p} \quad\text{ if }\quad 2+\frac{d}{p'}-\frac{d}{p}>0.
 \end{align}\\
  \emph{Step 2. $a$ satisfies \cref{Assumption-pde} $(H^a)$, $b=0$ and $f\in\tilde L_p$.} \\
  
 Here we apply the freezing coefficient argument.  For $\delta>0$ which will be determined later, let $\xi^\delta(\cdot):=\xi(\frac{\cdot}{\delta})$ where $\xi$ is the same function which we used to define 
 the localized spaces. For $z \in \R^d$ denote 
 $$\xi^{z,\delta}(x):=\xi^\delta(x-z),\quad a^z:=a(z), \quad u^z(x):=\xi^{z,\delta}(x)u(x),\quad f^z(x):=\xi^{z,\delta}(x)f(x).$$
 Observe that 
 \begin{align*}
     \lambda u^z-a_{ij}^z\partial_{ij}u^z=h_z
 \end{align*}
 where
 \begin{align*}
     h_z:=&f^z+(a_{ij}\partial_{ij}u )\xi^{z,\delta}-a_{ij}^z\partial_{ij}u^z
     \\=&f^z+(a_{ij}-a_{ij}^z)\partial_{ij}u\cdot\xi^{z,\delta}-a_{ij}^z(\partial_iu\partial_j\xi^{z,\delta}+\partial_ju\partial_i\xi^{z,\delta}+u\partial_{ij}\xi^{z,\delta}).
 \end{align*}
From  \cite[p18, 2. Corollary]{Krylov}, we know that  there exists some $N_0>0$ such that for any $\bar u \in H^{2,p}$ and $\epsilon>0$ we have
 $$\Vert \nabla \bar u\Vert_{L_p}\leq \epsilon \Vert\nabla^2\bar u\Vert_{L_p}+N_0\epsilon^{-1}\Vert  \bar u\Vert_{L_p}.$$
 Therefore
 \begin{align}\label{hz}
     \Vert h_z\Vert_{L_p}\leq& C_4\big(\Vert f^z\Vert_{L_p}+\omega_\alpha(a)\delta^\alpha\Vert \nabla^2 u\cdot\xi^{z,\delta}\Vert_{L_p}+K_2\Vert \nabla u\cdot\nabla\xi^{z,\delta}\Vert_{L_p}+K_2\Vert  u\cdot\nabla^2\xi^{z,\delta}\Vert_{L_p}\big)
      \nonumber  \\\leq& C_4\big(\Vert f^z\Vert_{L_p}+2\omega_\alpha(a)\delta^\alpha\Vert \nabla^2( u\cdot\xi^{z,\delta})\Vert_{L_p}+(K_2+2\omega_\alpha(a)\delta^\alpha)\Vert \nabla u\cdot\nabla\xi^{z,\delta}\Vert_{L_p}
      \nonumber\\&\quad\quad\quad\quad+(K_2+2\omega_\alpha(a)\delta^\alpha)\Vert  u\cdot\nabla^2\xi^{z,\delta}\Vert_{L_p}\big)
   \nonumber  \\\leq& C_5\big(\Vert f^z\Vert_{L_p}+2\omega_\alpha(a)\delta^\alpha\Vert \nabla^2 u^z\Vert_{L_p}+(K_2+2\omega_\alpha(a)\delta^\alpha)\delta^{-1}\Vert \nabla u\cdot\xi^{z,\delta}\Vert_{L_p}\nonumber\\&\quad\quad\quad\quad+(K_2+2\omega_\alpha(a)\delta^\alpha)\delta^{-2}\Vert  u\cdot\xi^{z,\delta}\Vert_{L_p}\big)\nonumber  \\\leq& C_5\big(\Vert f^z\Vert_{L_p}+2\omega_\alpha(a)\delta^\alpha\Vert \nabla^2 u^z\Vert_{L_p}+(K_2+2\omega_\alpha(a)\delta^\alpha)\delta^{-1}(\Vert \nabla  u^z\Vert_{L_p}+\Vert  u\cdot\nabla\xi^{z,\delta}\Vert_{L_p})
      \nonumber\\&\quad\quad\quad\quad+(K_2+2\omega_\alpha(a)\delta^\alpha)\delta^{-2}\Vert  u\cdot\xi^{z,\delta}\Vert_{L_p}\big)
    \nonumber \\\leq& C_6\big(\Vert f^z\Vert_{L_p}+(2\omega_\alpha(a)\delta^\alpha+\epsilon (K_2+2\omega_\alpha(a)\delta^\alpha)\delta^{-1}) \Vert \nabla^2u^z\Vert_{L_p}
   \nonumber\\&\quad\quad\quad\quad +(K_2+2\omega_\alpha(a)\delta^\alpha)(
    \epsilon^{-1}\delta^{-1}+\delta^{-2})\Vert  u\cdot\xi^{z,\delta}\Vert_{L_p}),
 \end{align}
 where $\omega_\alpha(a)$ is from \eqref{a-delta}.
 Assuming (without loss of generality) that $C_6\ge 1/6$, we define 
 \begin{align}
     \label{def:delta}\delta:= \Big(\frac{K_1}{6C_6(K_1+2\omega_\alpha(a))}\Big)^{1/\alpha}<1,\quad \epsilon:=\frac{K_1\delta}{6C_6(K_2+2\omega_\alpha(a)\delta^\alpha)}.
 \end{align}
 It is easy to see that $C_6\frac{1}{K_1}(2\omega_\alpha(a)\delta^\alpha+\epsilon (K_2+2\omega_\alpha(a)\delta^\alpha)\delta^{-1})<\frac{1}{2}$, and   $$(K_2+2\omega_\alpha(a)\delta^\alpha)(\epsilon^{-1}\delta^{-1}+\delta^{-2})\leq C_7\frac{K_2^2}{K_1}(\frac{K_1+\omega_\alpha(a)}{K_1})^{\frac{2}{\alpha}}.$$ So we get from \eqref{est-resl} and \eqref{hz} that 
 %for $\lambda\geq 1$
 \begin{align}\label{lambda-uz-lp}
     \Vert \nabla^2u^z\Vert_{L^{p}}\leq C_8\frac{1}{K_1}(\Vert f^z\Vert_{L_p}+\frac{K_2^2}{K_1}(\frac{K_1+\omega_\alpha(a)}{K_1})^{\frac{2}{\alpha}}
     \Vert  u^z\Vert_{L_p}).
     \end{align}
     Plugging this into  \eqref{hz} yields  $$  \Vert h_z\Vert_{L_p}\leq C_6\Big(\Vert f^z\Vert_{L_p}+\frac{C_8}{2C_6}\Big(\Vert f^z\Vert_{L_p}+\frac{K_2^2}{K_1}(\frac{K_1+\omega_\alpha(a)}{K_1})^{\frac{2}{\alpha}}\Vert  u^z\Vert_{L_p}\Big)+C_7\frac{K_2^2}{K_1}(\frac{K_1+\omega_\alpha(a)}{K_1})^{\frac{2}{\alpha}}\Vert  u^z\Vert_{L_p}\Big).$$
 Using the second inequality  in \eqref{est-resl}  we  get for $1+\frac{d}{p'}-\frac{d}{p}>0$
     \begin{align}\label{lambda-nablauz}
     \lambda^{(1+\frac{d}{p'}-\frac{d}{p})/2} \Vert \nabla u^z\Vert_{L^{p'}}\leq&  C_9{K_1}^{(\frac{d}{p'}-\frac{d}{p}-1)/2}\Big(\Vert f^z\Vert_{L_p}+\frac{K_2^2}{K_1}(\frac{K_1+\omega_\alpha(a)}{K_1})^{\frac{2}{\alpha}}\Vert  u^z\Vert_{L_p}\Big).
     \end{align}
     Similarly, for $2+\frac{d}{p'}-\frac{d}{p}>0$
     \begin{align}\label{lambda-uz}
      \lambda^{(2+\frac{d}{p'}-\frac{d}{p})/2} \Vert u^z\Vert_{L_{p'}}\leq& C_{10}{K_1}^{(\frac{d}{p'}-\frac{d}{p})/2}\Big(\Vert f^z\Vert_{L_p}+\frac{K_2^2}{K_1}(\frac{K_1+\omega_\alpha(a)}{K_1})^{\frac{2}{\alpha}}\Vert  u^z\Vert_{L_p}\Big).
 \end{align}
 Let $p'=p$.  Then
 \begin{align}\label{global-estimate}
 \lambda\Vert  u^z\Vert_{L_p}\leq C_{10}\Big(\Vert f^z\Vert_{L_p}+\frac{K_2^2}{K_1}(\frac{K_1+\omega_\alpha(a)}{K_1})^{\frac{2}{\alpha}}\Vert  u^z\Vert_{L_p}\Big).
 \end{align}
Taking $\lambda \geq 2 C_{10}\frac{K_2^2}{K_1}(\frac{K_1+\omega_\alpha(a)}{K_1})^{\frac{2}{\alpha}}=:C_{10}\kappa$   we obtain
$$\Vert  u^z\Vert_{L_p}\leq \frac{C_{10}}{\lambda -C_{10}\frac{K_2^2}{K_1}(\frac{K_1+\omega_\alpha(a)}{K_1})^{\frac{2}{\alpha}}}\Vert  f^z\Vert_{L_p},\quad \frac{K_2^2}{K_1}(\frac{K_1+\omega_\alpha(a)}{K_1})^{\frac{2}{\alpha}}\Vert  u^z\Vert_{L_p}\leq \Vert  f^z\Vert_{L_p}.$$ 
Together with \eqref{lambda-uz-lp}, \eqref{lambda-uz},  and \eqref{lambda-nablauz}, we have
\begin{align}
    \label{est:u-delta-z}
      \Vert \nabla^2u^z\Vert_{L^{p}}\leq& C_{12}\frac{1}{K_1}\Vert f^z\Vert_{L_p},\nonumber\\
     \lambda^{(1+\frac{d}{p'}-\frac{d}{p})/2} \Vert \nabla u^z\Vert_{L^{p'}}\leq&  C_{13}{K_1}^{(\frac{d}{p'}-\frac{d}{p}-1)/2}\Vert f^z\Vert_{L_p},\quad \text{ if }\quad 1+\frac{d}{p'}-\frac{d}{p}>0,\nonumber\\
      \lambda^{(2+\frac{d}{p'}-\frac{d}{p})/2} \Vert u^z\Vert_{L_{p'}}\leq& C_{14}{K_1}^{(\frac{d}{p'}-\frac{d}{p})/2}\Vert f^z\Vert_{L_p},\quad \text{ if }\quad 2+\frac{d}{p'}-\frac{d}{p}>0.
\end{align}
 From definition \eqref{chi} we know that, for each $z \in \R^d$, $\Vert  u^z\Vert_{L_p}\leq \Vert  u\Vert_{\tilde L_p}\lesssim\delta^{-d}\sup_{\bar z}\Vert  u^{\bar z}\Vert_{L_p}$\footnote{Recall that in \cref{notation} we assumed that the localized spaces are defined using the function $\xi^1$}, so we
 get
 %\begin{align*}
 %\Vert  u\Vert_{\tilde L_p}&\leq  \frac{C_{11}}{\lambda -C_{10}\frac{K_2^2}{K_1}(\frac{K_1+\omega_\alpha(a)}{K_1})^{\frac{2}{\alpha}}}\delta^{-d}\Vert  f\Vert_{\tilde L_p}\leq\frac{C_{12}}{\lambda -C_{10}\frac{K_2^2}{K_1}(\frac{K_1+\omega_\alpha(a)}{K_1})^{\frac{2d}{\alpha}}}(\frac{K_1+\omega_\alpha(a)}{K_1})^{d/\alpha}\Vert  f\Vert_{\tilde L_p} \\&\leq C_{13}\Vert  f\Vert_{\tilde L_p}.    
 %\end{align*}
 from   \eqref{est:u-delta-z} that 
for any $\lambda\geq C_{10} \kappa$ we have
 \begin{align}\label{est-resl1}
   \lambda^{(2+\frac{d}{p'}-\frac{d}{p})/2} \Vert  u\Vert_{\tilde L_{p'}} &\leq C_{15}{K_1}^{(\frac{d}{p'}-\frac{d}{p})/2}\delta^{-d}\Vert  f\Vert_{\tilde L_p}\quad\text{ if }\quad 2+\frac{d}{p'}-\frac{d}{p}>0,\nonumber\\
   \lambda^{(1+\frac{d}{p'}-\frac{d}{p})/2} \Vert \nabla u\Vert_{\tilde L_{p'}}&\leq  \lambda^{(1+\frac{d}{p'}-\frac{d}{p})/2} \sup_z(\Vert \nabla u^z\Vert_{ L_{p'}}+\Vert u\nabla \xi^{z,1}\Vert_{ L_{p'}})\nonumber\\&\leq  C_{16}({K_1}^{(\frac{d}{p'}-\frac{d}{p}-1)/2}+\lambda^{-1/2}{K_1}^{(\frac{d}{p'}-\frac{d}{p})/2})\delta^{-d}\Vert  f\Vert_{\tilde L_p}
   \nonumber\\&\leq  C_{17}{K_1}^{(\frac{d}{p'}-\frac{d}{p}-1)/2}\delta^{-d}\Vert  f\Vert_{\tilde L_p} \quad\text{ if }\quad 1+\frac{d}{p'}-\frac{d}{p}>0,   \nonumber\\
    \Vert \nabla^2u\Vert_{\tilde L^{p}}&\leq  \sup_z(\Vert \nabla^2 u^z\Vert_{ L_{p}}+\Vert u\nabla^2 \xi^{z,1}\Vert_{ L_{p}}+2\Vert \nabla u\nabla \xi^{z,1}\Vert_{ L_{p}})\nonumber\\&\leq  C_{18}\Big(\frac{1}{K_1}+\lambda^{-1}
    +\lambda^{-1/2}{K_1}^{-1/2}\Big)\delta^{-d}\Vert f\Vert_{\tilde L_p})
    \nonumber\\&\leq  C_{19}\frac{1}{K_1}\delta^{-d}\Vert f\Vert_{\tilde L_p}.
 \end{align}

  \emph{Step 3. $a$ is H\"older continuous and \cref{Assumption-pde} $(H^a)$ holds,  $|b|\in\tilde L_{p_1}$ and $f\in\tilde L_p$.}\\
  
  By \eqref{est-resl1} and H\"older's inequality, we have for $\lambda\geq C_{10}\kappa$  and $1+\frac{d}{p'}-\frac{d}{p}>0$
  \begin{align*}
      \lambda^{(1+\frac{d}{p'}-\frac{d}{p})/2} \Vert \nabla u\Vert_{\tilde L_{p'}}\leq &C_{17}{K_1}^{(\frac{d}{p'}-\frac{d}{p}-1)/2}\delta^{-d}\Vert  f+b\cdot\nabla u\Vert_{\tilde L_p}
      \\\leq &C_{17}{K_1}^{(\frac{d}{p'}-\frac{d}{p}-1)/2}\delta^{-d}(\Vert  f\Vert_{\tilde L_p}+\Vert  b\Vert_{\tilde L_{p_1}}\Vert \nabla u  \Vert_{\tilde L_{p_2}})
  \end{align*}
  where $p_1,p_2\in(p,\infty)$ and $\frac{1}{p_1}+\frac{1}{p_2}=\frac{1}{p}$.
  Let $p'=p_2$. Then we get
  \begin{align*}
      \lambda^{(1-\frac{d}{p_1})/2}\Vert \nabla u  \Vert_{\tilde L_{p_2}}\leq& C_{20}{K_1}^{(-\frac{d}{p_1}-1)/2}\delta^{-d}(\Vert  f\Vert_{\tilde L_p}+\Vert  b\Vert_{\tilde L_{p_1}}\Vert \nabla u  \Vert_{\tilde L_{p_2}}).
  \end{align*}
  Choosing $\lambda$ so large such that $$ \lambda^{(1-\frac{d}{p_1})/2}\geq C_{20}{K_1}^{\frac{-d/p_1-1}{2}}\delta^{-d}\Vert  b\Vert_{\tilde L_{p_1}},$$
  we get
  \begin{align*}
      \Vert \nabla u  \Vert_{\tilde L_{p_2}} \leq \frac{C_{20}{K_1}^{(-\frac{d}{p_1}-1)/2}}{\lambda^{(1-\frac{d}{p_1})/2}- C_{20}{K_1}^{\frac{-d/p_1-1}{2}}\delta^{-d}\Vert  b\Vert_{\tilde L_{p_1}}}\delta^{-d}\Vert  f\Vert_{\tilde L_p}.
  \end{align*}
  Moreover,
  \begin{align*}
      &\Vert b\cdot \nabla u\Vert_{\tilde L_p}\leq \frac{C_{20}
    {K_1}^{(-\frac{d}{p_1}-1)/2}\delta^{-d}\Vert  b\Vert_{\tilde L_{p_1}}}{\lambda^{(1-\frac{d}{p_1})/2}-C_{20}{K_1}^{(-\frac{d}{p_1}-1)/2}\delta^{-d}\Vert  b\Vert_{\tilde L_{p_1}}}\Vert  f\Vert_{\tilde L_p}=:\gamma\Vert  f\Vert_{\tilde L_p}.
  \end{align*}
%  It also yields out that
 % \begin{align*}
  %    C_{11}\Vert  b\Vert_{\tilde L_{p_1}}\Vert u  \Vert_{\tilde H^{1,p_2}}\leq  C_{11}C_{13}\Vert  b\Vert_{\tilde L_{p_1}}\Vert  f\Vert_{\tilde L_p}.
  %\end{align*}
  Using  \eqref{est-resl1} we see that  for any $\lambda$ such  that $\lambda\geq C_{10}\kappa$ and $ \lambda^{(1-\frac{d}{p_1})/2}\geq C_{20}{K_1}^{\frac{-1-d/p_1}{2}} \delta^{-d}\Vert  b\Vert_{\tilde L_{p_1}},$
  we have
  \begin{align*}
 \Vert \nabla^2u\Vert_{\tilde L_{p}}&\leq C_{21}(1+\gamma) \delta^{-d}\frac{1}{K_1}\Vert f\Vert_{\tilde L_p},\\ \lambda^{(1+\frac{d}{p'}-\frac{d}{p})/2} \Vert \nabla u\Vert_{\tilde L_{p'}}&\leq C_{22}(1+\gamma){K_1}^{(\frac{d}{p'}-\frac{d}{p}-1)/2} \delta^{-d}\Vert  f\Vert_{\tilde L_p}\quad\text{ if }\quad 1+\frac{d}{p'}-\frac{d}{p}>0,\\
  \lambda^{(2+\frac{d}{p'}-\frac{d}{p})/2} \Vert  u\Vert_{\tilde L_{p'}} &\leq C_{23}(1+\gamma){K_1}^{(\frac{d}{p'}-\frac{d}{p})/2} \delta^{-d}\Vert f\Vert_{\tilde L_p}\quad\text{ if }\quad 2+\frac{d}{p'}-\frac{d}{p}>0.
  \end{align*}
 Define   $  C_{24}:= \big(2C_{10}\big)\vee C_{20}$. Then,
  for $\lambda\geq C_{24}\kappa$ and $C_{24}\lambda^{-(1-\frac{d}{p_1})/2} {K_1}^{(-\frac{d}{p_1}-1)/2} \delta^{-d}\Vert  b\Vert_{\tilde L_{p_1}}<\frac{1}{2}$ (i.e. $\lambda \geq C_{24} K_1(\delta^{-d}\frac{\Vert b\Vert_{\tilde L_p}}{K_1})^{\frac{2}{1-d/p_1}}$) by taking $\lambda\geq C_{24}K_1\Big(\frac{K_2^2}{K_1^2}(\frac{K_1+\omega_\alpha(a)}{K_1})^{\frac{2}{\alpha}}+(\frac{K_1+\omega_\alpha(a)}{K_1})^{\frac{d}{\alpha}\frac{2}{1-d/p_1}}(\frac{\Vert b\Vert_{\tilde L_p}}{K_1})^{\frac{2}{1-d/p_1}}\Big)$ we get that there exists finite positive constant $C_{25}$  such that $ 1+\gamma\leq C_{25},$
   which finally shows the desired result \eqref{aprior-estimate-pde}  after plugging in the value of $\delta$ from \eqref{def:delta}.
 \end{proof}

 \begin{corollary}\label{homo}
 Let \cref{Assumption-pde} hold and $f=b^i,i=1,\cdots,d$ in \eqref{PDE}, let $p'\in[1,\infty]$. There exists some $C_0>0$ depending on $\alpha$, $p_1$ and $d$ only,  such that if we choose $\lambda\geq C_0K_1\Big(\frac{K_2^2}{K_1^2}(\frac{K_1+\omega_\alpha(a)}{K_1})^{\frac{2}{\alpha}}+(\frac{K_1+\omega_\alpha(a)}{K_1})^{\frac{d}{\alpha}\frac{2}{1-d/p_1}}(\frac{\Vert b\Vert_{\tilde L_p}}{K_1})^{\frac{2}{1-d/p_1}}\Big)$
 then for the solution $u^i$ to equation \eqref{PDE}    we have
 \begin{align}\label{gradu-infty}
   \Vert \nabla u^i\Vert_{\tilde L_{p'}}\leq \frac{1}{2}\lambda^{-\frac{d}{2p'}}K_1^{\frac{d}{2p'}}\leq\frac{1}{2}\quad\text{ if }\quad 1+\frac{d}{p'}-\frac{d}{p}>0,\nonumber\\  \Vert u\Vert_{\tilde L_{p'}}\leq \frac{1}{2}\lambda^{-\frac{1+d/p'}{2}}K_1^\frac{1+d/p'}{2}\leq\frac{1}{2} \quad\text{ if }\quad 2+\frac{d}{p'}-\frac{d}{p}>0.
 \end{align}
 \end{corollary}
\begin{proof}
Notice that for such $\lambda$   we have 
$C_0\lambda^{-(1-\frac{d}{p_1})/2} {K_1}^{(-\frac{d}{p_1}-1)/2} (\frac{K_1+\omega_\alpha(a)}{K_1})^{\frac{d}{\alpha}}\Vert  b\Vert_{\tilde L_{p_1}}<\frac{1}{2}$,
so  by \eqref{aprior-estimate-pde} for $f=b^i$,  
\begin{align*}
  \Vert \nabla u^i\Vert_{\tilde L_{p'}}\leq C\lambda^\frac{-1-d/p'+d/p_1}{2}{K_1}^{\frac{-1-d/p_1+d/p'}{2}}(\frac{K_1+\omega_\alpha(a)}{K_1})^{\frac{d}{\alpha}}\Vert b^i\Vert_{\tilde L_{p_1}}\leq\frac{1}{2}\lambda^{-\frac{d}{2p'}}K_1^{\frac{d}{2p'}}\leq\frac{1}{2}.
\end{align*}
With the similar argument we get $\Vert u\Vert_{\tilde L_{p'}}\leq C\lambda^\frac{-2-d/p'+d/p_1}{2}{K_1}^{\frac{d/p'-d/p_1}{2}}(\frac{K_1+\omega_\alpha(a)}{K_1})^{\frac{d}{\alpha}}\Vert b\Vert_{\tilde L_{p_1}}\leq \frac{1}{2}\lambda^{-\frac{1+d/p'}{2}}K_1^\frac{1+d/p'}{2}.$
\end{proof}
\section*{Acknowledgments}
\footnotesize Inspiring suggestion from and fruitful discussions with Benjamin Gess (Bielefeld) are  appreciated. Discussions with Xicheng Zhang (Beijing) and Zimo Hao (Bielefeld) are  acknowledged. CL is  supported by the DFG through the research unit (Forschergruppe) FOR 2402 and  the Austrian Science Fund (FWF) via the project "Regularisation by noise in discrete and continuous systems" . 


\begin{thebibliography}{999}
  \bibitem{A} L. Arnold: Random Dynamical Systems. {\it  Springer, Berlin} (1998).
 
%  \bibitem{AS} L. Arnold and M. Scheutzow: Perfect cocycles through stochastic differential equations. {\it Probab. Theory Relat. Fields} {\bf 101} (1995) 65-88.

\bibitem{CS} I. Chueshov and M. Scheutzow: On the structure of attractors and invariant measures for a class of monotone random systems. {\it Dyn.~Syst.} {\bf 19} (2004) 127-144. 
\bibitem{CSS1} M. Cranston, M. Scheutzow and D. Steinsaltz: Linear expansion of isotropic Brownian flows,
{\it Elect. Comm. in Probab.} {\bf 4} (1999) 91-101.  
\bibitem{CSS2} M. Cranston, M. Scheutzow and D. Steinsaltz: Linear bounds for stochastic dispersion,
{\it Ann.  Probab.} {\bf 28} (2000) 1852-1869.
\bibitem{CDS} H. Crauel, G. Dimitroff and M. Scheutzow: Criteria for strong and weak random attractors. {\it J. Dynamics and Diff. Equations},  {\bf 21} (2009) 233-247. 
\bibitem{CF} H. Crauel and F. Flandoli: Attractors for random dynamical systems. {\it  Probab. Theory Relat. Fields} {\bf 100} (1994) 365-393.
\bibitem{DS} G. Dimitroff and M. Scheutzow: Attractors and expansion for Brownian flows. {\it Electronic J. Probab.} {\bf 16} (2011) 1193-1213.

\bibitem{FGS1} F. Flandoli, B. Gess and M. Scheutzow: Synchronization by noise. {\it Probab. Theory Relat. Fields} {\bf {168}} (2017) 511-556.
\bibitem{FGS2} F. Flandoli, B. Gess and M. Scheutzow: Synchronization by noise for order-preserving random dynamical systems. {\it Ann. Probab.} {\bf {45}} (2017) 1325-1350.
\bibitem{Gess}B. Gess: Random attractors for stochastic porous media equations perturbed by space-time linear multiplicative noise. {\it Ann. Probab.}
{\bf 42} (2014) 818–864.
%\bibitem{GLR}B. Gess, W. Liu and M. R\"ockner: Random attractors for a class of stochastic partial differential equations driven by general additive noise. {\it J. Differential Equations}  {\bf 251} (2011) 1225-1253.
\bibitem{GLS}B. Gess, W. Liu and A. Schenke: Random attractors for locally monotone stochastic partial differential equations. {\it J. Differential Equations}  {\bf 269} (2020) 414-3455.

%\bibitem{GT} B. Gess and P. Tsatsoulis: Lyapunov exponents and synchronisation by noise for systems of SPDEs. {\it https://arxiv.org/pdf/2207.09820} (2022).
  \bibitem{GL}L. Galeati and C. Ling: Stability estimates for singular SDEs and applications. {\it Arxiv preprint} https://arxiv.org/abs/2208.03670. (2022).
%\bibitem{KS} G. Kager and M. Scheutzow: Generation of one-sided random dynamical systems by stochastic differential equations. {\it Electronic J. Probab.} {\bf 2} (1997) 1-8.
  \bibitem{Krylov}   N. V. Krylov: Lectures on Elliptic and Parabolic Equations in Sobolev spaces. {\it American Mathematical Society} (2008).
 \bibitem{KR}   N. V. Krylov  and M. R\"ockner: Strong solutions of stochastic equations with singular time dependent drift.
		{\it Probab. Theory Relat. Fields} {\bf{131}} (2005) 154-196.
\bibitem{KNS} C. Kuehn, A. Neamtu and S. Sonner: Random attractors via pathwise mild solutions for parabolic stochastic evolution equations. {\it Journal of Evolution Equations}  {\bf 21}  (2021) 2631–2663.

\bibitem{Le} K. L\^e: Quantitative John--Nirenberg inequality for stochastic processes of bounded mean oscillation. {\it Arxiv preprint} https://arxiv.org/pdf/2210.15736. (2022)
\bibitem{LL} K. L\^e and C. Ling: Taming singular stochastic differential equations: A numerical method.  {\it Arxiv preprint} https://arxiv.org/abs/2110.01343. (2021).
%\bibitem{LS} X. Li and  M. Scheutzow: Lack of strong completeness for stochastic flows. {\it Ann.  Probab.} {\bf 39}  (2011) 1407-1421.
  \bibitem{LSV} C. Ling, M. Scheutzow and I. Vorkastner: The perfection of local semi-flows and local random dynamical systems with applications to SDEs. {\it Stoch. Dyn.} {\bf 22} (2022).

%       \bibitem{MoS1} S. Mohammed and  M. Scheutzow: The stable manifold theorem for stochastic ordinary differential equations. {\it Ann.  Probab.} {\bf 27}  (1999) 615-652.

%        \bibitem{MS1} M. Scheutzow: Noise can create periodic behavior and stabilize nonlinear diffusions. {\it Stoch. Proc. Appl.} {\bf 20} (1985) 323-331.
 
\bibitem{LS1} H. Lisei and M. Scheutzow: Linear bounds and Gaussian tails in a stochastic dispersion model.
{\it Stoch. Dyn.}  {\bf 1} (2001) 389-403. 

\bibitem{LS2} H. Lisei and M. Scheutzow, M.: 
On the dispersion of sets under the action of an isotropic Brownian flow. 
In: {\it Proceedings of the Swansea 2002 Workshop Probabilistic Methods in Fluids}, World Scientific (2003) 224-238.

\bibitem{MS} M. Scheutzow: Chaining techniques and their application to stochastic flows. In: {\it Trends in Stochastic Analysis, eds: Blath, J., M\"orters, P., Scheutzow, M., Cambridge University Press}, 35-63 (2009).  
\bibitem{MS13} M. Scheutzow: A stochastic Gronwall lemma. {\it Infin. Dimens. Anal. Quantum Probab. Relat. Top.}
{\bf 16} (2013) 1350019, 4p.
%\bibitem{SV} M. Scheutzow and I. Vorkastner: Synchronization, Lyapunov exponents and stable manifolds for random dynamical systems. {\it  Stochastic Partial Differential Equations and Related Fields, eds: Eberle, A., Grothaus, M., Hoh, W., Kassmann, M., Stannat, W., Trutnau, G., Springer Proceedings in Mathematics $\&$ Statistics, Springe} (2018) 359-366.

 \bibitem{VE} A. Yu.  Veretennikov: On the strong solutions of stochastic differential equations. {\it Theory Probab. Appl.}  {\bf 24} (1979), 354-366.
 
  \bibitem{XXZZ}P. Xia, L.  Xie,  X. Zhang   and G. Zhao: $L^q(L^p)$-theory of stochastic differential equations. {\it Stochastic Process. Appl.} {\bf 130} (2020) 5188-5211.
 \bibitem{XieZhang2016} L. Xie and X. Zhang: Sobolev differentiable flows of SDEs with local Sobolev and super-linear growth coefficients. {\it Ann. Probab.} {\bf 22} (2016) 3661-3687.
       \bibitem{XieZhang2017} L.  Xie and   X. Zhang: Ergodicity of stochastic differential equations with jumps and singular coefficients. {\it Ann. Inst. H. Poincaré Probab. Statist.} {\bf 56} (2020) 175-229.
\bibitem{Zhang2011} X. Zhang: Stochastic homeomorphism flows of SDE with singular drifts and Sobolev diffusion coefficients.  {\it Electron. J. Probab.} {\bf 16} (2011) 1096-1116.
\bibitem{Zhang2017} X. Zhang:  Stochastic differential equations with Sobolev diffusion and singular drift and applications.  {\it Ann. Appl. Probab.} {\bf 26} (2016) 2697–2732.
%Xicheng Zhang. Stochastic dierential equations with Sobolev diusion and singular drift and applications.
%Ann. Appl. Probab., 26(5):2697–2732, 2016.
    \bibitem{Zhangzhao2018} X. Zhang and G. Zhao: Singular Brownian diffusion processes. {\it Commun.~Math.~Stat.} {\bf 6} (2018) 533-581.
     \bibitem{Zhangzhao-NS} X. Zhang and G. Zhao: Stochastic Lagrangian path for Leray’s solutions of 3D Navier–Stokes equations. {\it Commun. Math. Phys} {\bf 381} (2021) 491-525.

\bibitem{ZZ} R. Zhu and X. Zhu: Random attractor associated with the quasi-geostrophic equation. {\it Journal of Dynamics and Differential Equations} {\bf 29} (2017) 289-322.
\bibitem{Zv} A. K.  Zvonkin: A transformation of the phase space of a diffusion process that removes the drift. {\it Math. Sbornik} {\bf 135} (1974) 129-149.
\end{thebibliography}
	\end{document}